\documentclass[a4paper,american,reqno]{amsart}

\newif\ifPreprint \Preprintfalse
\newif\ifSubmission \Submissionfalse
\newif\ifAppendix \Appendixfalse
\newif\ifChallengeVersion \ChallengeVersiontrue
\newif\ifFullLengthVersion \FullLengthVersionfalse

\usepackage{babel}
\usepackage[utf8]{inputenc}
\usepackage[T1]{fontenc}
\usepackage{siunitx}
\usepackage{transparent}
\usepackage{booktabs}
\usepackage{multirow}
\usepackage{xspace}
\usepackage{relsize}
\usepackage{accents}
\usepackage{todonotes}
\usepackage{csvsimple}
\usepackage{csquotes}
\usepackage{microtype}
\usepackage{amsfonts}
\usepackage{pdflscape}
\usepackage{csquotes}
\usepackage{bbold}
\usepackage[ruled,linesnumbered]{algorithm2e}
\usepackage[style=numeric,
            giveninits=true,
            maxbibnames=100,
            maxcitenames=2,
            isbn=false,
            doi=true,
            url=false,
            isbn=false,
            backend=biber]{biblatex}

\usepackage[colorlinks,
            citecolor=green!50!black,
            urlcolor=blue,
            linkcolor=green!50!black]{hyperref} % Change back to blue for final version

\makeatletter
\patchcmd{\@settitle}{\uppercasenonmath\@title}{\scshape\large}{}{}
\patchcmd{\@setauthors}{\MakeUppercase}{\scshape\normalsize}{}{}
\makeatother

\vfuzz \hfuzz
\makeatletter
\@namedef{subjclassname@2020}{%
  \textup{2020} Mathematics Subject Classification}
\makeatother

\newcommand{\rev}[1]{#1}
\newcommand{\revtwo}[1]{#1}

\newcommand{\abbr}[1][abbrev]{#1.\xspace}

\newcommand{\eg}{\abbr[e.g]}

\newcommand{\ie}{\abbr[i.e]}

\newcommand{\wrt}{\abbr[w.r.t]}

\newcommand{\field}{\mathbb}
\newcommand{\naturals}{\field{N}}

\newcommand{\reals}{\field{R}}

\newtheorem{assumption}{Assumption}

\newtheorem{remark}{Remark}
\newtheorem{theorem}{Theorem}
\newtheorem{definition}{Definition}

\newtheorem{corollary}{Corollary}
\newtheorem{proposition}{Proposition}

\providecommand{\sfnamesize}{\relsize{0}}
\providecommand{\sfname}[1]{\textsf{\sfnamesize#1}\xspace}

\newcommand{\Gurobi}{\sfname{Gurobi}}
\newcommand{\Python}{\sfname{Python}}

\def\argmax{\mathop{\rm argmax}}

\newcommand{\bigM}{M}

\newcommand{\st}{\text{s.t.}}

\makeatletter
\newcommand{\fcdot}{\,\cdot\,}
\newcommand{\fcarg}[1]{\def\fc@rg{#1}\ifx\fc@rg\empty\fcdot\else\fc@rg\fi}
\makeatother
\newcommand{\abs}[1]{\lvert\fcarg{#1}\rvert}

\newcommand{\card}{\abs}

\newcommand{\norm}[2][]{\lVert\fcarg{#2}\rVert\ifx#1\empty\else_{#1}\fi}
\newcommand{\Norm}[2][]{\left\lVert#2\right\rVert\ifx#1\empty\else_{#1}\fi}

\newcommand{\snorm}[2][]{\lvert\!\lvert\!\lvert
  \fcarg{#1}\rvert\!\rvert\!\rvert\ifx#2\empty\else_{#1}\fi}
\newcommand{\Snorm}[2][]{\left\lvert\!\left\lvert\!\left\lvert
  #2\right\rvert\!\right\rvert\!\right\rvert\ifx#1\empty\else_{#1}\fi}
\newcommand{\sprod}[3][]{%
  \langle\fcarg{#2},\fcarg{#3}\rangle\ifx#1\empty\else_{#1}\fi}
\newcommand{\Sprod}[3][]{%
  \left\langle\fcarg{#2},\fcarg{#3}\right\rangle\ifx#1\empty\else_{#1}\fi}

\newcommand{\optmathindex}[1]{\ifx#1\empty\else,#1\fi}
\newcommand{\opttextindex}[1]{\ifx#1\empty\else,\text{#1}\fi}
\newcommand{\optmathsb}[1]{\ifx#1\empty\else_{#1}\fi}
\newcommand{\opttextsb}[1]{\ifx#1\empty\else_{\text{#1}}\fi}
\newcommand{\optmathsp}[1]{\ifx#1\empty\else^{#1}\fi}
\newcommand{\opttextsp}[1]{\ifx#1\empty\else^{\text{#1}}\fi}

\newcommand{\continuousFunctions}[1]{\mathcal{C}\ifx#1\empty\else^{#1}\fi}

\newcommand{\piecewiseContinuousFunctions}[1]{\mathcal{C}_p\ifx#1\empty\else^{#1}\fi}

\newcommand{\define}{\mathrel{{\mathop:}{=}}}

\newcommand{\myforall}[1]{\text{for all } #1}

\newcommand{\objval}{v}
\newcommand{\objfun}{f}
\newcommand{\cost}{c}
\newcommand{\decvar}{x}
\newcommand{\decvarset}{\mathcal{X}}
\newcommand{\probset}{X}
\newcommand{\proba}{q}
\newcommand{\sumproba}{\tilde{q}}
\newcommand{\scenproba}{p}
\newcommand{\scenario}{s}
\newcommand{\scenarios}{S}
\newcommand{\partition}{P}
\newcommand{\partitions}{\mathcal{P}}

\DeclareMathAlphabet{\mathbbold}{U}{bbold}{m}{n}
\newcommand*{\boldone}{\mathbbold{1}}
\newcommand{\indvar}{z}
\newcommand{\threshold}{\tau}
\newcommand{\uncrel}{\xi}
\newcommand{\uncrelset}{\Xi}

\newcommand{\LB}{\text{L}}
\newcommand{\UB}{\text{U}}
\newcommand{\ubar}[1]{\text{\b{$#1$}}}
\newcommand{\ubsol}{\bar{x}}
\newcommand{\lbsol}{\ubar{x}}
\newcommand{\iter}{j}

\newcommand{\infeasible}{\text{I}}
\newcommand{\feasible}{\text{F}}
\newcommand{\infscenarios}{\scenarios_{\infeasible}}
\newcommand{\feasscenarios}{\scenarios_{\feasible}}
\newcommand{\infpartitions}{\partitions_{\infeasible}}

\newcommand{\feaspartitions}{\partitions_{\feasible}}
\newcommand{\linkvar}{\pi}

\newcommand{\varmat}{A}
\newcommand{\constmat}{b}
\newcommand{\ineqsys}{G}
\newcommand{\entry}{i}
\newcommand{\bestbigM}{\bar{\bigM}}

\newcommand{\extrasize}{\mu}

\newcommand{\consindex}{\entry}
\newcommand{\consindexset}{I}
\newcommand{\subpart}{P}

\newcommand{\intervar}{\alpha}
\newcommand{\dualvar}{\eta}

\newcommand{\singlecost}{\rho}
\newcommand{\permutation}{\phi}
\newcommand{\quantile}{\text{Q}}
\newcommand{\permindex}{l}
\newcommand{\comptime}{\text{T}}
\newcommand{\iterset}{\text{It}}

\newcommand{\avg}{\text{a}}
\newcommand{\avtime}{\comptime_{\avg}}
\newcommand{\aviterset}{\iterset_{\avg}}
\newcommand{\avpartsize}{\card{\partition}_{\avg}}

\newcommand{\MILP}{\text{MILP}}

\newcommand{\textbigM}{\text{M}}
\newcommand{\Part}{\text{P}}
\newcommand{\random}{\text{random}}
\newcommand{\init}{\text{init}}

\newcommand{\violmerge}{\text{infeas}}
\newcommand{\final}{\text{final}}

\newcommand{\refined}{\mathcal{R}}
\newcommand{\merged}{\mathcal{M}}
\newcommand{\newsets}{N}
\newcommand{\infmerged}{\merged_{\infeasible}}
\newcommand{\infrefined}{\refined_{\infeasible}}
\newcommand{\feasmerged}{\merged_{\feasible}}
\newcommand{\feasrefined}{\refined_{\feasible}}

\newcommand{\newrefined}{R}
\newcommand{\newrefinedleft}{\newrefined_{\textup{left}}}
\newcommand{\newrefinedright}{\newrefined_{\textup{right}}}
\newcommand{\newmerged}{M}

\newcommand{\probability}{\field{P}}
\newcommand{\APM}{APM\xspace}
\newcommand{\divided}{\text{div}}
\newcommand{\projset}{E}
\newcommand{\newproba}{\kappa}

\newcommand{\feasregion}{\mathcal{F}}
\newcommand{\objregion}{\mathcal{O}}
\newcommand{\sizeparam}{\delta}
\newcommand{\subpartitions}{\tilde{\partitions}}
\newcommand{\partitionsleft}{\partitions_1}
\newcommand{\partitionsright}{\partitions_2}
\newcommand{\mergeind}{i}

\usepackage{pgfplots}
\usepackage{pgfplotstable}
\usepgfplotslibrary{groupplots} % for several plot on same figure
\usepgfplotslibrary{fillbetween}
\pgfplotsset{compat=1.16}
\pgfplotsset{grid = major, grid style={gray!30!white}}
\definecolor{darkgreen}{rgb}{0.31, 0.47, 0.26}

\tikzset{mipStyle/.style={red, mark=none, densely dashed, line width=1.5pt}}
\tikzset{randomStyle/.style={darkgreen, mark=none, dashdotted,
                             mark options = {solid}, line width=1.5pt}}
\tikzset{infeastyle/.style={brown, mark=none, dotted,
                            mark options = {solid}, line width=1.5pt}}
\tikzset{finalStyle/.style={blue, mark=none, line width=1.5pt}}

\newcommand{\plotresultsupperlower}[3]{
    \begin{tikzpicture}
        \begin{groupplot}[
            group style={
                group name=my plots,
                group size=2 by 1,
                xlabels at=edge bottom,
                ylabels at=edge left,
                horizontal sep=1.5cm
            },
            height = 5cm,
            width  = 6cm,
            xlabel = {Time (\unit{\second})},
            xmin = 0,
            xmax = #3,
            xtick = {0, #3/4, #3/2, 3*#3/4, #3},
            every y tick scale label/.style={at={(0,1)}, above left, inner sep=0pt, yshift=0.3em},
            x tick label style={/pgf/number format/precision=1},
            ]

            \nextgroupplot[title = {(a) Bounds}, ylabel = {Bound value}, font = \small, ymax = -5900, ymin = -6300,
                           legend style={at={(1.2,-0.35)},anchor=north}, legend columns=3]
            \addplot+[mipStyle, forget plot] table [x index = {0}, y index = {1}, col sep=comma]{#1};
            \addplot+[mipStyle] table [x index = {0}, y index = {2}, col sep=comma]{#1};
            \addlegendentry{\texttt{Song Big-M}}
            \addplot+[finalStyle, forget plot] table [x index = {1}, y index = {2}, col sep=comma]{#2};
            \addplot+[finalStyle] table [x index = {1}, y index = {3}, col sep=comma]{#2};
            \addlegendentry{$\Part_{\final}$}

            \nextgroupplot[title = {(b) Iterations}, ylabel = {Iterations}, ymin=0, font = \small]
            \addplot+[finalStyle] table [x index = {1}, y index = {0}, col sep=comma]{#2};
        \end{groupplot}
    \end{tikzpicture}
}
\bibliography{references}

\AddToHook{cmd/@setemails/before}{\raggedright}

\begin{document}

\title[Adaptive Partitioning for Chance-Constrained Problems]{Adaptive Partitioning for Chance-Constrained Problems with Finite Support}
\author[M. Roland, A. Forel, T. Vidal]{Marius Roland, Alexandre Forel, Thibaut Vidal}
\address[M. Roland, A. Forel, T. Vidal]{André-Aisenstadt Pavillon, 2920 Tour Road, Montreal, Quebec H3T 1N8, Canada}
\email[M. Roland]{mmmroland@gmail.com}
\email[A. Forel]{alexandre.forel@polymtl.ca}
\email[T. Vidal]{thibaut.vidal@polymtl.ca}

\date{\today}

\begin{abstract}
    This paper studies chance-constrained stochastic optimization problems with finite support. It presents an iterative method that solves reduced-size chance-constrained models obtained by partitioning the scenario set. Each reduced problem is constructed to yield a bound on the optimal value of the original problem. We show how to adapt the partitioning of the scenario set so that our adaptive method returns the optimal solution of the original chance-constrained problem in a finite number of iterations. At the heart of the method lie two fundamental operations: refinement and merging. A refinement operation divides a subset of the partition, whereas a merging operation combines a group of subsets into one. We describe how to use these operations to enhance the bound obtained in each step of the method while preserving the small size of the reduced model. Under mild conditions, we prove that, for specific refinement and merge operations, the bound obtained after solving each reduced model strictly improves throughout the iterative process. Our general method allows the seamless integration of various computational enhancements, significantly reducing the computational time required to solve the reduced chance-constrained problems. The method's efficiency is assessed through numerical experiments on chance-constrained multidimensional knapsack problems. We study the impact of our method's components and compare its performance against other methods from the recent literature.
\end{abstract}

\keywords{Chance Constraints, Adaptive Partitioning, Stochastic Optimization}
\subjclass[2020]{90C15, 90C11}

%%% Local Variables:
%%% mode: latex
%%% TeX-master: "../main"
%%% End:

\maketitle

\section{Introduction}
\label{sec:introduction}

We consider Chance-Constrained Stochastic Programs (CCSPs). Solving a CCSP amounts to finding the optimal value of a decision variable vector $\decvar \in \decvarset \subseteq \mathbb{R}^{n}$ that minimizes an objective function $\objfun: \decvarset \rightarrow \mathbb{R}$. This decision variable has to satisfy a constraint that depends on an uncertain parameter  $\uncrel\in\uncrelset$ with a probability of $(1-\threshold)$. Here, the parameter $\tau \in [0, 1]$ represents the risk tolerance of the decision-maker. The CCSP reads
\begin{subequations}
  \label{eq:ccsp-original}
        \begin{align}
            \objval^* = \min_{\decvar \in \decvarset} \quad & \objfun(\decvar) \\
            \st \quad & \probability_\uncrel[\decvar \in \probset(\uncrel)]
              \geq 1-\threshold.
    \end{align}
\end{subequations}
CCSPs were first introduced in~\cite{charnes1958cost}. Such models are used in a variety of fields such as power systems~\cite{cho2023exact,Porras2023}, vehicle routing~\cite{errico2018vehicle}, finance~\cite{Cattaruzza2022exact}, and contextual optimization~\cite{rahimian2023data}.

More specifically, we focus on CCSPs whose uncertain parameters $\uncrel\in\uncrelset$ have finite support. The uncertain parameters belong to the set $\uncrelset \define \{\uncrel^\scenario: \scenario \in \scenarios\}$ where each $\xi^s \in \mathbb{R}^d$ is a multi-dimensional vector representing a single realization of the uncertain parameters with probability $\scenproba_\scenario$ and $\scenarios$ is a set of scenarios. When $\uncrel$ has finite support, Model~\eqref{eq:ccsp-original} can be reformulated as
\begin{subequations}
    \label{eq:cclp-generic}
        \begin{align}
            \objval^* = \min_{\decvar \in \decvarset} \quad & \objfun(\decvar) \\
            \st \quad & \sum_{\scenario \in \scenarios} \scenproba_\scenario \boldone
              \left( \decvar \in \probset^\scenario \right) \geq 1-\threshold, \label{eq:chance-cons}
    \end{align}
\end{subequations}
where $\boldone$ is the indicator function, and $\probset^\scenario = \probset (\uncrel^\scenario)$ is the set of feasible decisions for realization $\uncrel^\scenario$. Models of the type~\eqref{eq:cclp-generic} are for instance obtained when the generic CCSP~\eqref{eq:ccsp-original} is approximatied using as Sample Average Approximation (SAA). The SAA method is a widely adopted approach, particularly when the distribution of $\uncrel$ is unknown~\cite{ahmed2014solving,pagnoncelli2009sample}. This is primarily because of its simplicity and the theoretical guarantees it offers~\cite{luedtke2008sample}.

By introducing a binary variable $\indvar_\scenario\in \{0,1\} $ for each $\scenario\in\scenarios$, Model~\eqref{eq:cclp-generic} can be reformulated as
\begin{subequations}
    \label{eq:cclp-reform}
    \begin{align}
        \objval^* = \min_{\decvar \in \decvarset} \quad & \objfun(\decvar) \\
        \st \quad & \indvar_\scenario = \boldone (\decvar \in \probset^\scenario),
          \quad \scenario \in \scenarios, \label{eq:indicator-var}\\
        & \sum_{\scenario \in \scenarios} \scenproba_\scenario\indvar_\scenario \geq 1-\threshold,
          \label{eq:chance-cons2}\\
        & \indvar_\scenario \in \{0,1\}, \quad \scenario \in \scenarios \label{eq:scenario-def}
    \end{align}
  \end{subequations}

Model~\eqref{eq:cclp-reform} is regularly solved by introducing so-called \enquote{big-M} coefficients. Big-M coefficients allow to reformulate the indicator constraints~\eqref{eq:indicator-var} without introducing additional variables or constraints. Let the feasible set of a scenario $\probset^\scenario$ be fully characterized by an inequality system $\ineqsys^\scenario(\decvar) \leq 0$, where $\ineqsys^\scenario:\reals^n \rightarrow \reals^m$. Further, assume that an upper bound $\bigM^\scenario_\entry$ exists on the maximum violation of the $\entry$-th constraint of scenario $\scenario$ for $\decvar$ in the set of feasible solutions of Model~\eqref{eq:cclp-reform}. Then, we may write
\begin{subequations}
    \label{eq:cclp-reform-bigm}
        \begin{align}
        \objval^* = \min_{\decvar \in \decvarset} \quad
        & \objfun(\decvar) \\
        \st \quad
        & \ineqsys^\scenario(\decvar)  \leq
          \bigM^\scenario (1-\indvar^\scenario),
          \quad \scenario \in \scenarios, \label{eq:indicator-var-bigm}\\
        & \sum_{\scenario \in \scenarios} \scenproba_\scenario\indvar_\scenario \geq 1-\threshold,
          \label{eq:chance-cons-bigm}\\
        & \indvar_\scenario \in \{0,1\}, \quad \scenario \in \scenarios, \label{eq:scenario-def-bigm}
    \end{align}
\end{subequations}
where $\bigM^\scenario$ is a vector of entries $\bigM^\scenario_\entry$. Large values for $\bigM^\scenario_\entry$ result in a loose continuous relaxation of Model~\eqref{eq:cclp-reform}, see, \eg,~\cite{qiu2014covering}. This is one of the reasons why optimization solvers \rev{designed to take advantage of continuous relaxations} struggle to solve chance-constrained problems \rev{even when their constraints and objective function are linear}. In addition, finding tight values for $\bigM^\scenario_\entry$ can be very time-consuming.

\subsection{Contributions}
% We make the following contributions:
\begin{enumerate}
    \item We propose an adaptive partitioning method for solving CCSPs with finite support to optimality. The method is based on iteratively solving reduced-size CCSPs that yield lower bounds on the optimal objective of the original CCSP. The partitions are adapted so that solutions obtained in previous iterations of the method are excluded from the feasible set of the reduced problem. By construction, the adaptive partitioning method terminates after a finite number of iterations and returns the optimal solution to the original CCSP.
    \item We specify the main operations of the method that are based on mathematical arguments. Namely, we study how to refine and merge an existing partition of the scenario set so that the lower bound obtained by solving a partitioned problem is guaranteed to strictly increase. This property has significant benefits when solving partitioned problems, it implies that big-M coefficients can be reused and tightened in successive iterations. This tightens the continuous relaxation of the reduced CCSPs that are solved and reduces the time needed to solve them. It further allows us to take increasing advantage of screening \cite{Porras2023} as the number of iterations of the method increases.
    \item We discuss how to create partitions that, by construction, result in a lower bound on the original optimal objective that is tighter than the well-studied quantile bound~\cite{ahmed2017nonanticipative}. In certain cases, this initial partition results in points that are feasible and hence optimal for the original CCSP.
    \item We compare the performance of the adaptive partitioning method for solving CCSPs with binary and continuous variables with state-of-the-art methods. The results demonstrate the computational advantage of the proposed method. We highlight why the proposed method performs well numerically by examining the effect of the main operations of the method.
\end{enumerate}

\subsection{Related Literature}
This paper lies at the interface between two branches of literature. The first branch concerns solving CCSPs with finite support. Different families of valid inequalities have been proposed for solving CCSPs. The intersection between well-studied mixing inequalities~\cite{gunluk2001mixing} and CCSPs have received a lot of attention~\cite{luedtke2014branch,luedtke2010integer,abdi2016mixing,kuccukyavuz2012mixing}. Quantile cuts and their relationship with mixing inequalities~\cite{xie2018quantile} have also been studied~\cite{qiu2014covering,kilincc2021joint}. In~\cite{tanner2010iis} the authors propose to solve CCSPs using a branch-and-cut approach that exploits cuts based on irreducibly infeasible subsystems of scenarios. Other approaches that do not use valid inequalities have also been proposed. In~\cite{ahmed2017nonanticipative} the authors study solution methods based on relaxations of the original CCSP. Moreover, they propose a scenario decomposition approach tailored to CCSPs with binary decision variables. The scenario decomposition approach is designed to iteratively solve the same relaxation of the original CCSP and add \enquote{no-good} cuts to exclude previously obtained solutions from the feasible set. In their conclusion, they highlight that a promising research direction is to investigate other techniques for excluding feasible solutions. This paper is in direct connection with our work since we take a similar direction for solving CCSPs. Moreover, our approach is independent of the types of variables that are considered. Further, some approaches are concerned with developing alternative problem formulations without binary variables~\cite{song2014chance,song2013branch}. The tightening of big-$\bigM$ parameters has also received much attention~\cite{Porras2023,qiu2014covering}.

The second branch of literature related to this paper is scenario reduction for solving stochastic problems with large scenario sets. We differentiate between methods that carry out the scenario reduction separately from the optimization process and the approaches that consider both aspects simultaneously. A priori scenario reduction is performed by minimizing the Wasserstein distance~\cite{heitsch2003scenario,dupavcova2003scenario} between the original and a smaller size scenario set. This idea is extended by some authors by incorporating objective function information in the scenario reduction problem~\cite{morales2009scenario,bertsimas2022optimization}. Recently, in~\cite{rujeerapaiboon2022scenario} the authors prove tight bounds on the approximation error of a reduced scenario set obtained via Wasserstein distance minimization. Combining scenario reduction and optimization is rather recent and has been applied to some stochastic optimization models. The generalized adaptive partition method (APM) for two-stage stochastic problems with fixed recourse in~\cite{song2015adaptive} has a strong link with our contribution. This method was initially proposed in~\cite{espinoza2014primal} for CVaR minimization and is based on the ideas presented in~\cite{bienstock2010solving}. Recently, in~\cite{ramirez2022generalized} the generalized APM ideas were extended to the case where the uncertain parameters follow a continuous distribution. Moreover, APMs have been combined with decomposition methods~\cite{pay2020partition}, stochastic dual dynamic programming~\cite{siddig2022adaptive}, and Benders decomposition~\cite{ramirez2023benders}.

To the best of our knowledge, APMs have not been proposed for chance-constrained problems, despite being strongly connected to the scenario decomposition approach presented in \cite{ahmed2017nonanticipative}. We bridge this gap by proposing an adaptive partitioning approach for solving general CCSPs with finite support.

\subsection{Outline}
Section~\ref{sec:adaptive} presents the general adaptive partitioning method and discusses how to obtain upper and lower bounds on the original problem objective by solving partitioned problems. Section~\ref{sec:refinement} studies how to efficiently refine partitions by excluding existing feasible solutions. Section~\ref{sec:merge} shows similar results for merge operations. Section~\ref{sec:strong-part} presents how to create partitions with a guaranteed tight lower bound on the objective of the original CCSP. Section~\ref{sec:strong-part} presents practical strategies that enhance the behavior of the algorithm. Section~\ref{sec:numerical-results} evaluates the computational performance of our algorithm and compares it with state-of-the-art methods on classical instances of the CCSP literature. Finally, Section~\ref{sec:conclusion} summarizes our findings and suggests directions for future work.

%%% Local Variables:
%%% mode: latex
%%% TeX-master: "../main"
%%% End:

\section{Adaptive Partitioning Method}
\label{sec:adaptive}
Our approach is based on creating partitions of the scenarios set, that is, grouping scenarios into subsets. We first formally define a partition. Then, we present two reduced-size models obtained via a partitioning of the scenario set. Both models yield bounds on the optimal objective of the original CCSP. Afterwards, we discuss how reduced-size models obtained via partitioning allow to compute tight big-$M$ parameters for Model~\eqref{eq:cclp-reform-bigm}. Finally, we explain the \APM and discuss its finite termination.

\begin{definition}
  A partition $\partitions = \left\lbrace \subpart_1, \subpart_2, \dots, \subpart_{\card{\partitions}}  \right\rbrace$ is a collection of non-empty subsets of the scenario set $\scenarios$ such that $\bigcup_{\subpart\in\partitions} \subpart = \scenarios$ and $\subpart_i \rev{\cap} \subpart_j = \emptyset , \myforall \subpart_i,\subpart_j\in\partitions$.
\end{definition}

\subsection{Partitioned CCSPs}

We now present how a reduced size CCSP is obtained by leveraging a partition~$\partitions$ of a scenario set~$\scenarios$. By aggregating the constraints of all the scenarios in a subset, we can construct two smaller-size chance-constrained problems that yield an upper and a lower bound on the optimal objective of the original CCSP. In the reduced CCSP each subset $\partition \in \partitions$ represents a unique scenario and the feasible set for $\partition\in\partitions$ is~$\probset^\partition = \bigcap_{\scenario\in\subpart}\probset^\scenario$.

 In~\cite{ahmed2017nonanticipative} the authors present the following reduced-size CCSP,
\begin{subequations}
    \label{eq:cclp-part}
    \begin{align}
        \objval^\LB(\partitions) = \min_{\decvar \in \decvarset} \quad
        & \objfun(\decvar) \\
        \st \quad
        & \indvar_\subpart = \boldone (\decvar \in \probset^\partition),
          \quad \subpart \in \partitions, \label{eq:indicator-var-part}\\
        & \sum_{\subpart \in \partitions}
          \proba_\subpart\indvar_\subpart \geq \sum_{\subpart \in
          \partitions} \proba_\subpart-\threshold, \label{eq:chance-cons-part}\\
        & \indvar_\subpart \in \{0,1\}, \quad \subpart \in \partitions, \label{eq:partition-def}
    \end{align}
\end{subequations}
where for each subset $\subpart\in\partitions$ the probability $\proba_\subpart = \min_{\scenario \in \subpart} \scenproba_\scenario$.
\begin{proposition}[from~\cite{ahmed2017nonanticipative}]
  \label{prop:lb-partition}
  The partitioned model~\eqref{eq:cclp-part} is a relaxation of the CCSP~\eqref{eq:cclp-reform}, \ie, $\objval^\LB(\partitions)\leq\objval^*$.
\end{proposition}

Conversely, if we set the probability of each subset $\subpart\in\partitions$ to $\rev{\sumproba_\subpart} = \sum_{\scenario \in \subpart} \scenproba_\scenario$. Then, another reduced-size CCSP reads
\begin{subequations}
  \label{eq:cclp-part2}
  \begin{align}
      \objval^\UB(\partitions) = \min_{\decvar \in \decvarset} \quad
    & \objfun(\decvar) \\
    \st \quad
    & \indvar_\subpart = \boldone (\decvar \in \probset^\partition),
      \quad \subpart \in \partitions, \label{eq:indicator-var-part2}\\
    & \sum_{\subpart \in \partitions} \rev{\sumproba_\subpart}\indvar_\subpart \geq 1-\threshold, \label{eq:chance-cons-part2}\\
    & \indvar_\subpart \in \{0,1\}, \quad \subpart \in \partitions. \label{eq:partition-def2}
  \end{align}
\end{subequations}

\begin{proposition}
    \label{prop:ub-partition}
    The scenario grouping model~\eqref{eq:cclp-part2} is a restriction of the CCSP~\eqref{eq:cclp-reform}, \ie, $\objval^*\leq\objval^\UB(\partitions)$.
\end{proposition}
\begin{proof}
    Let $\decvar'\in\decvarset$ be any point inside the feasible set of Model~\eqref{eq:cclp-part2} and let $\indvar'_\subpart$ satisfy Constraint~\eqref{eq:indicator-var-part2}. We set $\indvar_\scenario = \indvar'_\subpart$ for all $\partition\in\partitions$ and $\scenario \in \subpart$. Then, Constraint~\eqref{eq:chance-cons-part2} yields $\sum_{\scenario \in \scenarios} \scenproba_\scenario\indvar_\scenario = \sum_{\subpart \in \partitions} \rev{\sumproba_\subpart}\indvar_\subpart \geq 1-\threshold$ and $\decvar'$ is feasible for Model~\eqref{eq:cclp-reform}.
\end{proof}

\subsection{Adaptive Partitioning Method} We now discuss how the \APM works. The lower bound model~\eqref{eq:cclp-part} and the upper bound model~\eqref{eq:cclp-part2} are core components of the \APM. The partition~$\partitions$ is modified so that tighter upper and lower bounds on the original CCSP are obtained. Algorithm~\ref{alg:adaptive-partitioning} describes the general idea of the \APM.

\begin{algorithm}[!ht]
    \caption{Adaptive Partitioning Method.}
    \label{alg:adaptive-partitioning}
    \DontPrintSemicolon
    \SetKwInOut{Input}{Input}
    \Input{scenario set $\scenarios$, stopping criterion $\varepsilon \in \left(0,1\right)$.}
    \SetKwInOut{Output}{Output}
    \Output{optimal solution $\decvar^*$ of Model~\eqref{eq:cclp-reform}.}
    \SetKwInOut{Initialize}{Initialize}
    \Initialize{$\iter \leftarrow 0$, $v^\UB \leftarrow +\infty$, $v^\LB \leftarrow -\infty$.}
    Design the first partition $\partitions^0$. \label{alg:line-partition-design}\\
    \While{$(v^\UB - v^\LB)/\rev{\abs{v^\UB}} \geq \varepsilon$ \label{alg:lin-while-cond}}{
        \tcc{Compute a tighter lower bound.}
        Find $\lbsol^\iter$, the solution of Model~\eqref{eq:cclp-part} for the partition $\partitions^\iter$.\\
        \If{$\objval(\lbsol^\iter) > v^\LB$}{
            Set $v^\LB \leftarrow \objval(\lbsol^\iter)$.
        }
        \tcc{Compute a tighter upper bound.}
        % I modified this step because I find it weird to introduce the upper bounding problem and not use it. We can discuss how to best introduce the projection step (either in the current section or in Section 6).
        % Project the point~$\lbsol^\iter$ inside the feasible set of Model~\eqref{eq:cclp-reform}, let $\ubsol^\iter$ denote this point \label{alg:line-projection}\\
        Find $\ubsol^\iter$, the solution of Model~\eqref{eq:cclp-part2} for the partition $\partitions^\iter$.\\
        \If{$\objval(\ubsol^\iter) < v^\UB$}{
            Set $x^\UB \leftarrow \ubsol^\iter$ and $v^\UB \leftarrow \objval(\ubsol^\iter)$.
        }
        \tcc{Modify the partition.}
        Modify $\partitions^{\iter}$ to obtain a new partition $\partitions^{\iter+1}$. \label{alg:line-partition-update}\\
        Increment iteration $\iter \leftarrow \iter + 1$.
    }
    \KwRet{$\decvar^\UB$}
  \end{algorithm}

%%% Local Variables:
%%% mode: latex
%%% TeX-master: "../main"
%%% End:

Algorithm~\ref{alg:adaptive-partitioning} contains three main steps: computing a lower bound, computing an upper bound, and modifying the partition. The modification of the partition has the biggest influence on the computational performance of the \APM. Typically, Algorithm~\ref{alg:adaptive-partitioning} exhibits a trade-off between solving a large number of computationally tractable CCSPs or solving a small number of computationally intensive CCSPs. The performance of Algorithm~\ref{alg:adaptive-partitioning} also depends on the design of the initial partition~$\partitions^0$ as it dictates the tightness of the initial bounds \wrt the optimal value~$\objval^*$. The initial partition~$\partitions^0$ also influences the structure of any subsequent partition.

\rev{Algorithm~\ref{alg:adaptive-partitioning} runs until a user-specified optimality gap is reached. This termination condition can be replaced by a time limit~$\comptime<\comptime_{\text{max}}$, where the parameters~$\comptime$ and~$\comptime_{\text{max}}$ represent the elapsed time and the maximum computation time, respectively. The time limit should be sufficiently large to find initial lower and upper bounds. When the time limit is reached, the algorithm can return the current best solution~$\decvar^\UB$ and the relative optimality gap $(v^\UB - v^\LB)/\abs{v^\UB}$.}

\subsection{Finite Termination}
As long as $\card{\partitions^\iter}$ increases over the course of the iterations, Algorithm~\ref{alg:adaptive-partitioning} terminates in a finite number of iterations and recovers the optimal solution of Model~\eqref{eq:cclp-reform}. Indeed, if $\card{\partitions^{\iter + 1}} > \card{\partitions^{\iter}}$ holds for all $\iter$, the partition $\partitions^\iter$ will eventually contain exactly one scenario per subset $\subpart \in \partitions^\iter$. In that case, Model~\eqref{eq:cclp-part} reduces to Model~\eqref{eq:cclp-reform}, and Algorithm~\ref{alg:adaptive-partitioning} returns the optimal solution of Model~\eqref{eq:cclp-reform}. The worst-case number of iterations needed by Algorithm~\ref{alg:adaptive-partitioning} to recover the optimal solution of Model~\eqref{eq:cclp-reform} is $\card{\scenarios} - \card{\partitions^0}$. However, the aim when applying Algorithm~\ref{alg:adaptive-partitioning} is to avoid reaching the point where~$\partitions^\iter$ equals~$\scenarios$ because in that case there is no computational advantage of using the~\APM.

Our presentation of Algorithm~\ref{alg:adaptive-partitioning} allows for flexibility in designing each of its main steps. For instance, computing the upper bound can be performed by a heuristic that does not necessarily solve a reduced size problem obtained via a partitioning of the scenarios. In the following, we present several methods to design and modify partitions that preserve finite termination of the method. \rev{These methods are based on refinements that split a scenario subset into two new subsets. These operations are discussed in detail in the next section. In particular, we show how to construct minimal-size refinements that guarantee a strict increase of the resulting lower bound.}

%%% Local Variables:
%%% mode: latex
%%% TeX-master: "../main"
%%% End:

\section{On Constructing Minimal Size Refinements}
\label{sec:refinement}

In this section, we present how we propose to modify a partition in Algorithm~\ref{alg:adaptive-partitioning}. We make the following assumption throughout the document.
\begin{assumption}
  \label{ass:equiprobable}
   All scenarios are equiprobable, \ie~$\scenproba_\scenario = 1/\card{\scenarios}$ for all~$\scenario\in\scenarios$.
\end{assumption}
Assumption~\ref{ass:equiprobable} is introduced mainly to simplify the notation. For instance, it always holds when scenarios are obtained via a sample-average approximation. Further, any scenario set with rational probabilities can be modified to satisfy Assumption~\ref{ass:equiprobable} by duplication of scenarios. Naturally, this increases the size of the scenario set but does not change the feasible set of Model~\eqref{eq:cclp-reform}.

\begin{remark}
    \label{remark:chance-cons-equi}
    If Assumption~\ref{ass:equiprobable} is satisfied, then Constraint~\eqref{eq:chance-cons-part} is equivalent to
    \begin{equation}
        \label{eq:condition-removal}
        \sum_{\partition\in\partitions} \indvar_\partition  \geq \abs{\partitions} - \lfloor \threshold \abs{\scenarios} \rfloor.
    \end{equation}
\end{remark}%
% - Detailed proof if needed -
% \begin{proof}
%     Since,~$\proba_\partition = \min_{\scenario\in\partition} 1/\abs{\scenarios}=1/\abs{\scenarios}$, Constraint~\eqref{eq:chance-cons-part} can be reformulated as
%     \begin{equation}
%         \label{eq:chance-cons-part-reform}
%         \sum_{\partition\in\partitions} \indvar_\partition  \geq \abs{\partitions} - \threshold \abs{\scenarios}.
%     \end{equation}
%     Further, since the variables~$\indvar_\partition$ are binary, Equation~\eqref{eq:chance-cons-part-reform} is equivalent to
%     \begin{equation*}
%         \sum_{\partition\in\partitions} \indvar_\partition  \geq \abs{\partitions} - \lfloor \threshold \abs{\scenarios} \rfloor.
%     \end{equation*}
% \end{proof}
\rev{Remark~\ref{remark:chance-cons-equi} is obtained by observing that Assumption ~\ref{ass:equiprobable} implies that $\proba_\partition = \min_{\scenario\in\partition} 1/\abs{\scenarios}=1/\abs{\scenarios}$ and that all the variables~$\indvar_\partition$ are binary.} Remark~\ref{remark:chance-cons-equi} highlights that any partition~$\partitions$ should satisfy~$\abs{\partitions} > \lfloor \threshold \abs{\scenarios} \rfloor$. Indeed, if~$\abs{\partitions} \le \lfloor \threshold \abs{\scenarios} \rfloor$, then Constraint~\eqref{eq:chance-cons-part} is always trivially satisfied and any~$\decvar\in\decvarset$ is feasible for Model~\eqref{eq:cclp-part} with~$\partitions$.

\subsection{Solution Exclusion by Refinement}
\rev{We carry out the partition modification step in the \APM using refinements as described in Definition~\ref{def:refinement}.} This is explained by the fact that \rev{successive refinement} operations guarantee finite termination of the \APM at the optimal solution of the original CCSP.

\begin{definition}
  \label{def:refinement}
    A partition~$\refined$ is a refinement of a partition~$\partitions$ if for any $ \newrefined \in \refined$ there exists a subset~$\partition \in \partitions $ such that $ \newrefined \subseteq \partition$.
\end{definition}

We now study how to perform efficient refinement operations. We refine a partition by excluding the solution obtained in the previous iteration of the \APM. We prove that if this solution is unique, then we can create a minimal size refinement~$\refined$ of~$\partitions$ such that~$\objval^{\LB}(\refined) > \objval^{\LB}(\partitions)$.

To carry out this proof, we proceed in the following way. First, in Proposition~\ref{prop:refined-relax}, we show that for any refinement~$\refined$ of a partition~$\partitions$, Model~\eqref{eq:cclp-part} with~$\partitions$ is a relaxation of Model~\eqref{eq:cclp-part} with~$\refined$. \rev{Second, in Proposition~\ref{prop:smallest-refinement}, we discuss how we can design refinements such that a specific feasible point of Model~\eqref{eq:cclp-part} with~$\partitions$, which is not feasible for Model~\eqref{eq:cclp-reform}, becomes infeasible for Model~\eqref{eq:cclp-part} with~$\refined$. Using the ideas of Proposition~\ref{prop:smallest-refinement} we introduce a simple procedure for producing such refinements in Algorithm~\ref{alg:refinement}. Third, in Proposition~\ref{prop:lower-bound-part}, we prove the validity of a lower bound on the number of refinements needed to exclude a feasible solution. Then, by combining Proposition~\ref{prop:smallest-refinement} and Proposition~\ref{prop:lower-bound-part} in Corollary~\ref{cor:min-size-ref}, we show that this bound is tight for any refinement obtained with Algorithm~\ref{alg:refinement}.} Finally, in Theorem~\ref{thm:strictly-incr-lower}, we use the previous results to show how a refinement~$\refined$ satisfying~$\objval^{\LB} ( \refined) > \objval^{\LB} ( \partitions)$ is obtained in Algorithm~\ref{alg:adaptive-partitioning}.

To simplify the presentation, we use the following notation throughout the remainder of this section. We let~$\partitions$ denote a partition of the scenario set~$\scenarios$ and~$\refined$ be a partition resulting from a refinement of~$\partitions$. Given any point~$\rev{\decvar}\in\decvarset$, we define~$\indvar_\partition(\rev{\decvar}) = \boldone (\rev{\decvar }\in \bigcap_{\scenario\in\partition}\probset^\scenario)$ the value of the indicator variables associated with a subset~$\subpart\in\partitions$, as well as~$\feaspartitions(x) \define \{\partition\in\partitions:\ \decvar\in\probset^\partition\}$ the set of subsets that are feasible for $x$ and~$\infpartitions(x) \define \{\subpart\in\partitions:\ \decvar\notin\probset^\partition\}$ the set of subsets of~$\partitions$ that are infeasible for $x$. Naturally,~$\partitions = \feaspartitions(x) \cup \infpartitions(x)$. \rev{Similar to $\infpartitions(x)$, for any point~$\rev{\decvar}\in\decvarset$, we define $\infscenarios (x) \define \{\scenario\in\scenarios:\ \decvar\notin\probset^\scenario\}$ the set of scenarios that are infeasible for $x$.}

\begin{proposition}
  \label{prop:refined-relax}
  \rev{If $\refined$ is a refinement of partition $\partitions$, then} Model~\eqref{eq:cclp-part} with~$\partitions$ is a relaxation of Model~\eqref{eq:cclp-part} with~$\refined$, \ie,
  \begin{equation*}
    \label{eq:lower-bound-refinement}
    \objval^* \geq \objval^{\LB}(\refined) \geq \objval^{\LB}(\partitions).
  \end{equation*}
\end{proposition}
\begin{proof}

  \rev{First, we reformulate the model of~$\objval^{\LB}(\refined)$ as a standalone CCSP. We introduce}
  \begin{equation*}
    \newproba_{\newrefined} = \frac{\proba_{\newrefined}}{\sum_{\newrefined' \in \refined}\proba_{\newrefined'}},\quad
    \varepsilon = \frac{\tau}{\sum_{{\newrefined'} \in \refined} \proba_{\newrefined'}}.
  \end{equation*}
  \rev{This yields the following CCSP,}
  \begin{subequations}
    \label{eq:reform-refined}
    \begin{align}
        \objval^\LB(\refined) = \min_{\decvar \in \decvarset} \quad
        & \objfun(\decvar) \\
        \st \quad
        & \indvar_\newrefined = \boldone (\decvar \in \probset^\newrefined),
          \quad \newrefined \in \refined, \\
        & \sum_{\newrefined \in \refined}
          \newproba_\newrefined\indvar_\newrefined \geq 1-\varepsilon, \\
        & \indvar_\newrefined \in \{0,1\}, \quad \newrefined \in \refined.
    \end{align}
  \end{subequations}
  \rev{This reformulation shows that $\refined$ can be considered as a standalone scenario set that parameterizes Model~\eqref{eq:reform-refined}. Each subset $\newrefined \in \refined$ can be seen as an individual scenario of the CCSP~\eqref{eq:reform-refined} with associated feasible set $\probset^\newrefined$ and probability $\newproba_\newrefined$.

  Since~$\refined$ is a refinement of~$\partitions$ and a standalone scenario set for Model~\eqref{eq:reform-refined} the partition~$\partitions$ can be seen as a partition of~$\refined$. We introduce the notation~$\partitions (\refined)$ and~$\partitions (\scenarios)$ to highlight the case in which~$\partitions$ is defined using $\refined$ or using $\scenarios$, respectively. By Proposition~\ref{prop:lb-partition}, we have~$\objval^{\LB}(\refined) \geq \objval^{\LB}(\partitions(\refined))$. We proceed by showing that the model of~$\objval^{\LB}(\partitions(\refined))$ is equivalent to the model of~$\objval^{\LB}(\partitions(\scenarios))$.} If this holds, then~$ \objval^{\LB}(\refined) \geq \objval^{\LB}(\partitions(\refined)) = \objval^{\LB}(\partitions(\scenarios)) = \objval^{\LB}(\partitions)$.

We write Constraint~\eqref{eq:chance-cons-part} for the model of~$\objval^\LB(\partitions(\refined))$, it reads,
  \begin{align*}
    \sum_{\subpart \in \partitions (\refined)} \proba_\subpart\indvar_\subpart &\geq \sum_{\subpart \in \partitions (\refined)} \proba_\partition-\varepsilon,\\
    \sum_{\subpart \in \partitions (\refined)} \min_{\newrefined \in \subpart} (\newproba_{\rev{\newrefined}}) \indvar_\subpart &\geq \sum_{\subpart \in \partitions (\refined)} \min_{\newrefined \in \subpart} (\newproba_{\rev{\newrefined}}) - \varepsilon,\\
    \sum_{\subpart \in \partitions (\refined)}
    \min_{\newrefined \in \subpart} (\min_{\scenario\in\newrefined}(\rev{\scenproba_{\scenario}})) \indvar_\subpart &\geq \sum_{\subpart \in\partitions (\refined)} \min_{\newrefined \in \subpart} (\min_{\scenario\in\newrefined} (\rev{\scenproba_{\scenario}})) - \threshold,
  \end{align*}
    which is equivalent to Constraint~\eqref{eq:chance-cons-part} for the model of~$\objval^{\LB}(\partitions(\scenarios))$. Similarly, we state Constraint~\eqref{eq:indicator-var} for the model of~$\objval^{\LB}(\partitions(\refined))$,
    \begin{equation*}
        \indvar_\partition = \boldone (\decvar \in \probset^{\partition}) = \boldone (\decvar \in \bigcap_{\newrefined\in\partition}\probset^{\newrefined}) = \boldone (\decvar \in  \bigcap_{\newrefined\in\partition} \bigcap_{\scenario\in\newrefined} \probset^{\scenario}),
    \end{equation*}
    which is equivalent to Constraint~\eqref{eq:indicator-var} for the model of~$\objval^{\LB}(\partitions(\scenarios))$. Hence,~$\objval^{\LB}(\partitions(\refined)) = \objval^{\LB}(\partitions( \scenarios))$ and Model~\eqref{eq:cclp-part} with~$\partitions$ is a relaxation of Model~\eqref{eq:cclp-part} with~$\refined$.
\end{proof}

\begin{remark}
  Proposition~\ref{prop:refined-relax} holds whether Assumption~\ref{ass:equiprobable} is satisfied or not.
\end{remark}

\begin{remark}
  By Proposition~\ref{prop:refined-relax} we know that any big-M coefficient obtained for a partition~$\partitions$ is valid for any refinement ~$\refined$ of~$\partitions$. This property allows to keep the same big-$\bigM$ values through the iterations of Algorithm~\ref{alg:adaptive-partitioning} if only refinements are carried out.
\end{remark}

Next, we show that \rev{for any feasible solution~$\lbsol$ of Model~\eqref{eq:cclp-part} with~$\partitions$ that is infeasible for Model~\eqref{eq:cclp-reform} we can construct a refinement $\refined$ of $\partitions$ such that~$\lbsol$ is infeasible for Model~\eqref{eq:cclp-part} with~$\refined$.}

\begin{proposition}
    \label{prop:smallest-refinement}
    Let the point~$\lbsol$ be feasible for Model~\eqref{eq:cclp-part} with~$\partitions$ but infeasible for Model~\eqref{eq:cclp-reform}. There exists a refinement~$\refined$ of~$\partitions$ with size~$\card{\refined} = \card{\partitions} + \extrasize$, where
    \begin{equation*}
        \label{eq:extrasize-bound}
        \extrasize = \lfloor \tau\card{\scenarios} \rfloor + 1 - \card{\infpartitions(\lbsol)},
    \end{equation*}
    such that~$\lbsol$ is infeasible for Model~\eqref{eq:cclp-part} with~$\refined$.
\end{proposition}
\begin{proof}
  For what follows, we fix the value of parameter~$\extrasize$ to~$\lfloor \tau\card{\scenarios} \rfloor + 1 - \card{\infpartitions(\lbsol)}$. In addition, we introduce the sets
  \begin{align*}
    \partitions_1(\lbsol) &= \{\partition\in\partitions: \ \card{\partition\cap \infscenarios(\lbsol)}=1\},\quad \scenarios_1(\lbsol) =   \infscenarios (\lbsol) \cap ( \cup_{\partition\in\partitions_1 (\lbsol)} \partition ),\\
    \partitions_{2}(\lbsol) &= \{\partition\in\partitions: \ \card{\partition\cap \infscenarios(\lbsol)}\geq2\}, \quad \scenarios_{2}(\lbsol) = \infscenarios (\lbsol) \cap ( \cup_{\partition\in\partitions_2 (\lbsol)} \partition ),
  \end{align*}
  \rev{We have~$\card{\partitions_1(\lbsol)} = \card{\scenarios_1(\lbsol)}$}.
  Since~$\lbsol$ is infeasible for Model~\eqref{eq:cclp-reform} we know that~$\card{\infscenarios(\lbsol)} \geq \lfloor \tau\card{\scenarios} \rfloor + 1$\rev{, this implies that}
  \begin{align}
    \extrasize
    & \leq \card{\infscenarios(\lbsol)} - \card{\infpartitions(\lbsol)}, \nonumber\\
    & \leq  \card{\scenarios_1(\lbsol)} + \card{\scenarios_2(\lbsol)} - \card{\partitions_1(\lbsol)} - \card{\partitions_2(\lbsol)},\nonumber\\
    & \rev{=} \card{\scenarios_2(\lbsol)} - \card{\partitions_2(\lbsol)}. \label{eq:split-possible}
  \end{align}
By Equation~\eqref{eq:split-possible} we know that there exist at least~$\extrasize$ scenarios~$\scenario \in \scenarios_2 (\lbsol)$ that can be removed from a set~$\partition \in \partitions_2(\lbsol)$ while \rev{ensuring} that~$\card{\infpartitions(\lbsol)}$ stays unchanged. \rev{We carry out this removal and create $\refined$ by giving birth to exactly~$\extrasize$ infeasible subsets.}

\rev{As a consequence, we have}
\begin{align}
  \card{\refined}
  &= \card{\partitions} + \extrasize, \nonumber\\
  &= \card{\feaspartitions(\lbsol)} +  \lfloor \tau\card{\scenarios} \rfloor + 1,\nonumber\\
  &= \sum _{\newrefined \in \refined} \indvar_\newrefined (\lbsol) + \lfloor \tau\card{\scenarios} \rfloor + 1, \label{eq:unrdored-infeasible}
\end{align}
where we used that~$\card{\partitions(\lbsol)} = \card{\infpartitions(\lbsol)} + \card{\feaspartitions(\lbsol)}$ and~$\sum _{\newrefined \in \refined} \indvar_\newrefined (\lbsol) = \card{\feasrefined(\lbsol)} = \card{\feaspartitions(\lbsol)} = \sum _{\partition \in \partitions} \indvar_\partition (\lbsol)$.
\rev{By reorganizing the terms in Equation~\eqref{eq:unrdored-infeasible}, we observe that $\lbsol$ is not feasible for Model~\eqref{eq:cclp-part} with~$\refined$.}
\end{proof}

The proof of Proposition~\ref{prop:smallest-refinement} describes a method that requires low computational effort for creating a refined partition~$\refined$ where~$\lbsol$ is excluded from the feasible set of Model~\eqref{eq:cclp-part} with~$\refined$.  Algorithm~\ref{alg:refinement} describes this method. As explained in the proof of Proposition~\ref{prop:smallest-refinement},~$\refined$ is created by splitting~$\extrasize$ subsets that contain at least two infeasible scenarios. Each split creates two new subsets that contain at least one infeasible scenario. \rev{We draw attention to the fact that in Line~\ref{alg:line-remaining-split} of Algorithm~\ref{alg:refinement}, the method for allocating the remaining scenarios is not specified. This omition is intentional, as this step is discussed in detail in Section~\ref{sec:accurate-split}. The next proposition describes a lower bound on the size of~$\refined$.}

\begin{algorithm}[ht]
  \caption{Minimal Size Refinement.}
  \label{alg:refinement}
  \DontPrintSemicolon
  \SetKwInOut{Input}{Input}
  \Input{scenario set $\scenarios$, partition $\partitions$ of $\scenarios$, point $\lbsol \in \decvarset$.}
  \SetKwInOut{Output}{Output}
  \Output{refinement $\refined$ of $\partitions$, where $\lbsol$ is not feasible for Model~\eqref{eq:cclp-part} with $\refined$.}
  \SetKwInOut{Initialize}{Initialize}
  \Initialize{$\refined \leftarrow \partitions$ and $\extrasize \leftarrow \lfloor \tau\card{\scenarios} \rfloor + 1 - \card{\infpartitions(\lbsol)}$.}
  \While{$\card{\refined} < \card{\partitions} + \extrasize$}{
    Select $\newrefined_1 \in \infrefined(\lbsol)$ such that $\card{\newrefined_1 \cap \infscenarios(\lbsol)} \geq 2$. \label{alg:line-refinement-selection}\\
    Select two infeasible scenarios $\scenario_1, \scenario_2 \in \newrefined_1 \cap \infscenarios(\lbsol)$.\\
    Set $\newrefinedleft \leftarrow \{\scenario_1\}$ and $\newrefinedright \leftarrow \{\scenario_2\}$. \label{alg:line-inf-split}\\
    Allocate all remaining scenarios $\scenario \in \newrefined_1 \setminus \{\scenario_1,\scenario_2\}$ to $\newrefinedleft$ and $\newrefinedright$. \label{alg:line-remaining-split}\\
    Set $\refined \leftarrow\refined \setminus \{\newrefined_1\} \cup \{\newrefinedleft,\newrefinedright\}$.
  }
  \KwRet{$\refined$}
\end{algorithm}

%%% Local Variables:
%%% mode: latex
%%% TeX-master: "../main"
%%% End:

\begin{proposition}
    \label{prop:lower-bound-part}
    Let the point~$\lbsol$ be feasible for Model~\eqref{eq:cclp-part} with~$\partitions$ but infeasible for Model~\eqref{eq:cclp-reform}. If a refinement~$\refined$ of~$\partitions$ is such that~$\lbsol$ is not feasible for Model~\eqref{eq:cclp-part} with~$\refined$, then
    \begin{equation*}
        \card{\refined}-\card{\partitions}\geq \extrasize = \lfloor \tau\card{\scenarios} \rfloor + 1 - \card{\infpartitions(\lbsol)}.
    \end{equation*}
\end{proposition}
\begin{proof}
  It is not possible to refine subsets~$\partition\in\feaspartitions(\lbsol)$ into subsets~$\newrefined\in\infrefined(\lbsol)$ using only split operations. Hence, for any refinement~$\refined$ of partition~$\partitions$ we have~$\card{\feasrefined (\lbsol)} \geq \card{\feaspartitions (\lbsol)}$.  By assumption,~$\lbsol$ is infeasible for Model~\eqref{eq:cclp-part} with~$\refined$, we have
  \begin{equation}
    \card{\refined} - \lfloor \tau\card{\scenarios} \rfloor - 1 \geq \sum_{\partition\in\refined} \indvar_\partition(\lbsol) = \card{\feaspartitions(\lbsol)} = \card{\partitions} - \card{\infpartitions(\lbsol)}.\label{eq:condition-mu}
  \end{equation}
  Finally, by reorganizing terms in Equation~\eqref{eq:condition-mu} we get
  \begin{equation*}
        \card{\refined}-\card{\partitions}\geq  \lfloor \tau\card{\scenarios} \rfloor + 1 - \card{\infpartitions(\lbsol)}.
    \end{equation*}
\end{proof}

\rev{Given a point $\lbsol$ that is feasible for Model~\eqref{eq:cclp-part} with~$\partitions$ and infeasible for Model~\eqref{eq:cclp-reform}, Proposition~\ref{prop:lower-bound-part} provides a lower bound on the size of any refinement $\refined$ that renders the point infeasible for Model~\eqref{eq:cclp-part} with~$\refined$.}

\rev{\begin{corollary}
  \label{cor:min-size-ref}
  Let the point~$\lbsol$ be feasible for Model~\eqref{eq:cclp-part} with~$\partitions$ but infeasible for Model~\eqref{eq:cclp-reform}. Let $\refined$ be a refinement of $\partitions$ obtained by applying Algorithm~\ref{alg:refinement}. Then, $\refined$ is the smallest size refinement of~$\partitions$ such that~$\lbsol$ is infeasible for Model~\eqref{eq:cclp-part} with~$\refined$.
\end{corollary}
\begin{proof}
  By Proposition~\ref{prop:smallest-refinement}, Algorithm~\ref{alg:refinement} generates a refinement $\refined$ of size $\card{\refined} = \card{\partitions} + \lfloor \tau\card{\scenarios} \rfloor + 1 - \card{\infpartitions(\lbsol)}$. Proposition~\ref{prop:lower-bound-part} confirms this is a lower bound.
\end{proof}}

\rev{Corollary~\ref{cor:min-size-ref} shows that Algorithm~\ref{alg:refinement} produces the smallest size refinement~$\refined$ of~$\partitions$ such that~$\lbsol$ is infeasible for Model~\eqref{eq:cclp-part} on~$\refined$.} The following theorem is the main result used in Algorithm~\ref{alg:adaptive-partitioning} for obtaining refinements~$\refined$ such that~$\objval^{\LB}(\refined) > \objval^{\LB}(\partitions)$.
\begin{theorem}
  \label{thm:strictly-incr-lower}
    Let the point~$\lbsol$ be an optimal solution of Model~\eqref{eq:cclp-part} with~$\partitions$. If~$\lbsol$ is unique and infeasible for Model~\eqref{eq:cclp-reform}, then the refinement~$\refined$ of~$\partitions$ created by applying Algorithm~\ref{alg:refinement} is the smallest size refinement of~$\partitions$ such that
    \begin{equation*}
        \label{eq:increasing-bound-split}
        \objval^{\LB}(\refined) > \objval^{\LB}(\partitions).
    \end{equation*}
\end{theorem}
\begin{proof}
  Proposition~\ref{prop:refinement-bound} proves that any refinement~$\refined$ of~$\partitions$ is such that Model~\eqref{eq:cclp-part} with~$\partitions$ is a relaxation of Model~\eqref{eq:cclp-part} with~$\refined$. Then, if~$\refined$ is created by applying Algorithm~\ref{alg:refinement} we know by Proposition~\ref{prop:smallest-refinement} that~$\lbsol$ is not feasible for Model~\eqref{eq:cclp-part} with~$\refined$. Moreover, if~$\lbsol$ is the unique point such that solving Model~\eqref{eq:cclp-part} with~$\partitions$ returns the value~$\objval^\LB (\partitions)$, it follows that
  \begin{equation*}
    \objval^{\LB}(\refined) > \objval^{\LB}(\partitions).
  \end{equation*}
  In addition, \rev{Corollary~\ref{coro:lbsol-feasible-part}} proves that~$\refined$ is of minimal size.
\end{proof}

%%% Local Variables:
%%% mode: latex
%%% TeX-master: "../main"
%%% End:

\section{On Constructing Minimal Size Mergers}
\label{sec:merge}
Adaptive partitioning methods based uniquely on refinement operations strictly increase the size of the partition~$\partitions$ with the iterations. As a consequence, the number of indicator variables~$\indvar$ and the time required to solve Model~\eqref{eq:cclp-reform-bigm} increase over time. We suggest adding merging operations to address this issue.

\begin{definition}
    \label{def:merge}
    A partition~$\merged$ of a set~$\scenarios$ is a merger of a partition~$\partitions$ if for any~$\partition \in \partitions$ there exists~$\newmerged \in \merged $ such that $ \partition \subseteq \newmerged$.
\end{definition}
In general, it is not possible to perform a merge without increasing the lower bound obtained by solving Model~\eqref{eq:cclp-part}. This is stated in the following corollary.

\begin{corollary}
    \label{prop:merge-worse}
    For any merger~$\merged$ of a partition~$\partitions$, we have~$\objval^{\LB}(\merged) \le \objval^{\LB}(\partitions)$.
\end{corollary}
\begin{proof}
  The proof is a direct consequence of Definition~\ref{def:merge}. If~$\merged$ is a merger of~$\partitions$ then~$\partitions$ is a refinement of~$\merged$ and Proposition~\ref{prop:refined-relax} holds.
\end{proof}
Corollary~\ref{prop:merge-worse} suggests that a merging operation should not be performed in Algorithm~\ref{alg:adaptive-partitioning} if we request a strict increasing lower-bound. \rev{Still, studying merging operations can provide both algorithmic benefits and structural insights. First, we show that an \enquote{a-posteriori} merge may be performed in the \APM, resulting in strictly tighter lower bounds. Second, we investigate the conditions under which mergers can be constructed such that~$\objval^{\LB}(\merged) = \objval^{\LB}(\partitions)$. This provides general insight into when small partitions that yield optimal solutions are likely to exist.}

\subsection{Solution Exclusion by Merging}
We now study how to perform efficient merging operations. We prove that, if certain conditions are satisfied, given a refinement~$\refined$ of a partition~$\partitions$, we can construct a merger~$\merged$ of~$\refined$ such that~$\objval^{\LB} ( \merged) > \objval^{\LB} ( \partitions)$.
To carry out this proof, we proceed in the following way. First, in Propositions~\ref{prop:merge-bound} and~\ref{prop:refinement-bound} we identify conditions that guarantee~$\objval^{\LB} ( \merged) \geq \objval^{\LB} ( \partitions)$. Second, in Corollary~\ref{coro:lbsol-feasible-part}, we show that, if additional assumptions are made on~$\merged$, then~$\objval^{\LB} ( \merged) > \objval^{\LB} ( \partitions)$. Third, in Proposition~\ref{prop:merger-construction}, we propose a procedure for creating a merger that satisfies the aforementioned conditions. Fourth, in Theorem~\ref{thm:strict-increase-merge}, we use the previous results to show how a merger~$\merged$ satisfying~$\objval^{\LB} ( \merged) > \objval^{\LB} ( \partitions)$ is obtained in Algorithm~\ref{alg:adaptive-partitioning}. \rev{Finally, in Proposition~\ref{prop:cycling} we show that any infeasible point that has an objective value smaller or equal to $\objval^{\LB} ( \partitions)$ cannot become feasible after a merge operation. This result implies that no cycling between two solutions with the same objective function value can occur when merge operations are considered.}

To simplify the presentation, we use the following notation \rev{throughout the paper, unless stated otherwise}. Let~$\partitions$ be a partition of the scenario set~$\scenarios$ and let~$\refined$ be a refinement of~$\partitions$. Let~$\merged$ be a merger of~$\refined$. Let~$\newsets^A$ be the subsets of an arbitrary partition $\mathcal{A}$ of $\scenarios$ that are not in~$\partitions$, i.e.,~$\newsets^{\mathcal{A}} = \mathcal{A} \setminus \partitions$. Further, let~$\singlecost_\partition$ denote the minimum cost of the solution that satisfies the constraint of a subset~$\partition \in \partitions$, \ie,
\begin{equation}
    \singlecost_\partition = \min_\decvar\{ \objfun(\decvar): \decvar \in  \probset^\partition\cap\decvarset\}. \label{eq:single-cost-definition}
  \end{equation}

We now identify conditions that ensure~$\objval^\LB (\merged) \geq \objval^\LB (\partitions)$.
\begin{proposition}
    \label{prop:merge-bound}
     \rev{Let $\refined$ be a refinement of a partition $\partitions$. If~$\merged$ is a merger of $\refined$ and}~$\card{\merged} \geq \card{\partitions}$, then
    \begin{equation*}
        \label{eq:merge-bound}
        \objval^{\LB} ( \merged) \geq \min \left\lbrace \objval^{\LB}(\partitions), \min_{\newmerged\in\newsets^\merged} \singlecost_\newmerged \right\rbrace.
    \end{equation*}
\end{proposition}
\begin{proof}
    Let~$\lbsol^\merged$ be the optimal solution of Model~\eqref{eq:cclp-part} with~$\merged$. We distinguish two cases.\\
        \textit{Case 1}. If~$\indvar_{\rev{\newmerged}}(\lbsol^\merged)=0$ for all~$\rev{\newmerged}\in\newsets^\merged$ we know that~$\lbsol^\merged$ is also feasible for Model~\eqref{eq:cclp-part} with partition~$\partitions$ and thus ~$\objval^\LB(\merged) \geq \objval^\LB(\partitions)$.\\
        \textit{Case 2}. If there exists a~$\rev{\newmerged}\in\newsets^\merged$ such that~$\indvar_{\rev{\newmerged}}(\lbsol^\merged)=1$ it holds that~$\objval^\LB(\merged) \geq \singlecost_\newmerged$.
\end{proof}

Furthermore, we present a simpler method to evaluate situations where a merge operation guarantees that the lower bound is non-decreasing.
\begin{proposition}
    \label{prop:refinement-bound}
    \rev{Let $\refined$ be a refinement of a partition $\partitions$. If~$\merged$ is a merger of $\refined$ and}~$\card{\merged} \geq \card{\partitions}$, then
    \begin{equation*}
        \objval^{\LB}(\merged) \geq \min \left\lbrace \objval^{\LB}(\partitions), \min_{\newrefined\in\newsets^{\refined}}\singlecost_\newrefined \right\rbrace.
    \end{equation*}
\end{proposition}
\begin{proof}
    Since the partition~$\merged$ is a merger of~$\refined$,
    \begin{equation*}
        \label{eq:inter-min-sing-cost}
        \min_{\newmerged\in\merged} \singlecost_\newmerged \geq \min_{\newrefined\in\refined} \singlecost_\newrefined.
    \end{equation*}
    It follows from Proposition~\ref{prop:merge-bound} that
    \begin{equation*}
        \objval^{\LB}(\merged)
        \geq \min (\objval^{\LB}(\partitions), \min_{\newmerged\in\newsets^{\merged}}(\singlecost_\newmerged)) \geq \min (\objval^{\LB}(\partitions), \min_{\newrefined\in\newsets^{\refined}}(\singlecost_\newrefined)).
    \end{equation*}
\end{proof}
Propositions~\ref{prop:merge-bound} and~\ref{prop:refinement-bound} describe straightforward conditions on the value that~$\singlecost_\partition$ may take to guarantee~$\objval^\LB (\merged) \geq \objval^\LB (\partitions)$. If~$\min_{\partition\in\newsets^{\refined}}\singlecost_\partition \geq \objval^\LB (\partitions)$, Proposition~\ref{prop:refinement-bound} highlights that~$\objval^\LB (\merged) \geq \objval^\LB (\partitions)$ for any~$\merged$ obtained by merging the refinement~$\refined$. Moreover, if~$\singlecost_\newrefined \leq \objval^\LB (\partitions)$ for at least one~$\newrefined \in \newsets^\refined$ it may be possible to construct a new partition~$\merged$ such that~$\min_{\newmerged\in\newsets^{\merged}}\singlecost_\newmerged \geq \objval^\LB (\partitions)$.
\begin{corollary}
  \label{coro:lbsol-feasible-part}
  Let~$\lbsol$ be the optimal solution of Model~\eqref{eq:cclp-part} with~$\partitions$. Let~$\merged$ be such that~$\card{\merged} \geq \card{\partitions}$. Let~$\min_{\newmerged\in\newsets^\merged} \singlecost_\newmerged > \objval^{\LB}(\partitions)$ or~$\min_{\newrefined\in\newsets^\refined} \singlecost_\newrefined > \objval^{\LB}(\partitions)$. If~$\lbsol$ is unique and infeasible for Model~\eqref{eq:cclp-part} with~$\merged$, then
  \begin{equation*}
    \objval^{\LB}(\merged) > \objval^{\LB}(\partitions).
  \end{equation*}
\end{corollary}
\begin{proof}

Let~$\lbsol^\merged$ be the optimal solution of Model~\eqref{eq:cclp-part} with~$\merged$. To prove the statement we consider two cases.\\
    \textit{Case 1}. Let~$\indvar_{\newmerged'}(\lbsol^\merged)=0$ for all~$\newmerged'\in\newsets^\merged$. We know that~$\lbsol^\merged$ is also feasible for Model~\eqref{eq:cclp-part} with partition~$\partitions$. However, the point~$\lbsol$ is unique and not feasible for Model~\eqref{eq:cclp-part} with partition~$\merged$, so~$\objval^\LB(\merged)$ cannot be equal to~$\objval^{\LB}(\partitions)$, and~$\objval^{\LB}(\merged) > \objval^{\LB}(\partitions)$. \\
    \textit{Case 2}. Let a subset~$\newmerged'\in\newsets^\merged$ exist such that~$\indvar_{\newmerged'}(\lbsol^\merged)=1$. If~$\min_{\newmerged\in\newsets^\merged} \singlecost_\newmerged > \objval^{\LB}(\partitions)$ then~$\objval^\LB(\merged)\geq \singlecost_{\newmerged'} \geq \min_{\newmerged\in\newsets^\merged} \singlecost_\newmerged > \objval^{\LB}(\partitions)$. Otherwise, if~$\min_{\newrefined\in\newsets^\refined} \singlecost_\newrefined > \objval^{\LB}(\partitions)$, then~$\objval^\LB(\merged) \geq \singlecost_{\newmerged'} \geq \min_{\newmerged\in\newsets^\merged} \singlecost_\newmerged \geq \min_{\newrefined\in\newsets^\refined} \singlecost_\newrefined > \objval^{\LB}(\partitions)$.
\end{proof}
 Corollary~\ref{coro:lbsol-feasible-part} states a set of conditions that allows to identify when a merge operation leads to a strictly increasing lower bound. Notice that Model~\eqref{eq:cclp-part2} does not need to be solved, instead it is sufficient to compute the value of~$\min_{\newmerged\in\newsets^\merged} \singlecost_\newmerged$ and~$\min_{\newrefined\in\newsets^\refined} \singlecost_\newrefined$.

We now show how to design a merger~$\merged$ such that~$\lbsol$ is not feasible for Model~\eqref{eq:cclp-part} with~$\merged$. Moreover, we want~$\merged$ to be of minimal size, \ie,~$\card{\merged} = \card{\partitions}$. The following proposition describes a valid way to create such a partition.
\begin{proposition}
    \label{prop:merger-construction}
    Let the point~$\lbsol$ be feasible for Model~\eqref{eq:cclp-part} with~$\partitions$ but infeasible for Model~\eqref{eq:cclp-reform} and let~$\refined$ be a refinement of the partition~$\partitions$ such that~$\card{\refined} = \card{\partitions} + \extrasize$, with~$\extrasize = \lfloor \tau\card{\scenarios} \rfloor + 1 - \card{\infpartitions(\lbsol)}$. Then, a merger~$\merged$ of the partition~$\refined$ satisfying~$\card{\merged} = \card{\partitions}$ exists such that the point~$\lbsol$ is not feasible for Model~\eqref{eq:cclp-part} with~$\merged$.
\end{proposition}
\begin{proof}
    First, we show that if~$\merged$ satisfies the following two properties,
    \begin{align}
        \card{\merged} &= \card{\partitions},\label{eq:properties-merger1}\\
        \card{\feasmerged(\lbsol)} &= \card{\feaspartitions(\lbsol)} - \extrasize, \label{eq:properties-merger2}
    \end{align}
    then~$\lbsol$ is not feasible for Model~\eqref{eq:cclp-part} with~$\merged$.
    By combining Equations~\eqref{eq:properties-merger1} and~\eqref{eq:properties-merger2} we have
    \begin{align*}
        \card{\feasmerged(\lbsol)} &= \card{\feaspartitions(\lbsol)} - \lfloor \tau\card{\scenarios} \rfloor - 1 + \card{\infpartitions(\lbsol)},\\
        &= \card{\partitions} - \lfloor \tau\card{\scenarios} \rfloor - 1,\\
        &= \card{\merged} - \lfloor \tau\card{\scenarios} \rfloor - 1,
    \end{align*}
    which, by Remark~\ref{remark:chance-cons-equi}, implies that~$\lbsol$ is not feasible for Model~\eqref{eq:cclp-part} with~$\merged$.

    Second, we describe how to construct a merger~$\merged$ of~$\refined$ that satisfies Equations~\eqref{eq:properties-merger1} and~\eqref{eq:properties-merger2}. Observe that it is not possible to refine subsets~$\partition\in\feaspartitions(\lbsol)$ into subsets~$\newrefined\in\infrefined(\lbsol)$ using split operations. Hence, for any refinement~$\refined$ of partition~$\partitions$ the inequality~$\card{\feasrefined (\lbsol)} \geq \card{\feaspartitions (\lbsol)}$ holds. We distinguish two cases.

    \textit{Case 1}. If~$\card{\partitions} = \lfloor \tau\card{\scenarios} \rfloor + 1$, then~$\extrasize=\card{\partitions} - \card{\infpartitions (\lbsol)} = \card{\feaspartitions (\lbsol)} \le \card{\feasrefined (\lbsol)}$. A valid merger~$\merged$ is obtained by merging~$\extrasize$ subsets in~$\feasrefined (\lbsol)$ together with exactly one subset in~$\infrefined (\lbsol)$. This is always possible because~$\feasrefined (\lbsol)$ contains at least~$\extrasize$ elements and~$\infrefined (\lbsol)$ is not empty since~$\lbsol$ would otherwise necessarily be feasible for Model~\eqref{eq:cclp-reform}. This yields a merger~$\merged$ of size~$\card{\partitions}$ satisfying Equation~\eqref{eq:properties-merger1}. Moreover, exactly~$\extrasize$ sets~$\newrefined \in \feasrefined(\lbsol )$ are merged with a set~$\newrefined' \in \infrefined(\lbsol)$ such that Equation~\eqref{eq:properties-merger2} is satisfied.

    \textit{Case 2}. If~$\card{\partitions} \geq \lfloor \tau\card{\scenarios} \rfloor + 2$, then~$\extrasize\leq\card{\partitions} - \card{\infpartitions (\lbsol)} - 1 = \card{\feaspartitions (\lbsol)} - 1 \le \card{\feasrefined (\lbsol)} - 1$. We create the merger~$\merged$ by merging~$\extrasize+1$ sets inside~$\feasrefined (\lbsol)$. This yields a merger~$\merged$ of size~$\card{\partitions}$ satisfying Equation~\eqref{eq:properties-merger1}. Moreover,~$\extrasize + 1$ sets are merged to form one subset~$\newmerged\in\feasmerged(\lbsol)$ and Equation~\eqref{eq:properties-merger1} is satisfied.
\end{proof}

\begin{algorithm}[!ht]
  \caption{Maximal Size Merger.}
  \label{alg:merge}
  \DontPrintSemicolon
  \SetKwInOut{Input}{Input}
  \Input{scenario set $\scenarios$; partition $\partitions$ of $\scenarios$; refinement $\refined$ of $\partitions$; point $\lbsol \in \decvarset$.}
  \SetKwInOut{Output}{Output}
  \Output{merger $\merged$ of $\refined$, where $\lbsol$ is not feasible for Model~\eqref{eq:cclp-part} with $\merged$.}
  \SetKwInOut{Initialize}{Initialize}
  \Initialize{$\merged \leftarrow \refined$ and $\extrasize \leftarrow \lfloor \tau\card{\scenarios} \rfloor + 1 - \card{\infpartitions(\lbsol)}$.}
  \uIf{$\card{\partitions} = \lfloor \threshold \card{\scenarios}\rfloor + 1 $}{
    Select all $\newmerged_1, \dots, \newmerged_\extrasize \in \feasmerged(\lbsol)$ and $\newmerged_{\extrasize + 1} \in \infmerged(\lbsol)$. \label{alg:line-merge-selection-1}\\
  }
  \Else{
   Select $\newmerged_1, \dots, \newmerged_{\extrasize + 1} \in \feasmerged(\lbsol)$. \label{alg:line-merge-selection-2}\\
  }
  Set $\newmerged_{\extrasize+2} \leftarrow \bigcup_{i \in [\extrasize + 1]} \newmerged_i$ and $\merged \leftarrow \merged \setminus \{\newmerged_1, \dots, \newmerged_{\extrasize + 1}\} \cup \{\newmerged_{\extrasize + 2}\}.$\\
  \KwRet{$\merged$}
\end{algorithm}

%%% Local Variables:
%%% mode: latex
%%% TeX-master: "../main"
%%% End:

Proposition~\ref{prop:merger-construction} gives a simple procedure for creating a maximum-size merger that excludes the point~$\lbsol$. Algorithm~\ref{alg:merge} summarizes how to carry out this procedure. The following theorem is the main result used in Algorithm~\ref{alg:adaptive-partitioning} for obtaining mergers~$\merged$ such that~$\objval^{\LB}(\merged) > \objval^{\LB}(\partitions)$.

\begin{theorem}
  \label{thm:strict-increase-merge}
  Let~$\lbsol$ be the optimal solution of Model~\eqref{eq:cclp-part} with~$\partitions$. Let~$\refined$ be a refinement of~$\partitions$ constructed by applying Algorithm~\ref{alg:refinement} such that~$\min_{\newrefined\in\newsets^{\refined}} \singlecost_\newrefined > \objval^{\LB}(\partitions)$. If~$\lbsol$ is unique and infeasible for Model~\eqref{eq:cclp-reform}, then the merger~$\merged$ of~$\refined$ obtained by applying Algorithm~\ref{alg:merge} is such that~$\card{\merged} = \card{\partitions}$ and
  \begin{equation*}
    \label{eq:strict-bound-merge}
    \objval^{\LB}(\merged) > \objval^{\LB}(\partitions).
  \end{equation*}
\end{theorem}
\begin{proof}
  The partition~$\refined$ is constructed by applying Algorithm~\ref{alg:refinement}. This means that~$\card{\refined} = \card{\partitions} + \extrasize$, with~$\extrasize = \lfloor \tau\card{\scenarios} \rfloor + 1 - \card{\infpartitions(\lbsol)}$. By applying Algorithm~\ref{alg:merge}, we construct a merger~$\merged$ as described in the proof of Proposition~\ref{prop:merger-construction}. Consequently, the point~$\lbsol$ is not feasible for Model~\eqref{eq:cclp-part} with~$\merged$ and~$\card{\merged} = \card{\partitions}$. Since~~$\lbsol$ is not feasible for Model~\eqref{eq:cclp-part} with~$\merged$ and~$\min_{\newrefined\in\newsets^{\newrefined}} \singlecost_\newrefined > \objval^{\LB}(\partitions)$ or~$\min_{\newmerged\in\newsets^{\merged}} \singlecost_\newmerged > \objval^{\LB}(\partitions)$ we know by Corollary~\ref{coro:lbsol-feasible-part} that~$\objval^\LB (\merged) > \objval^\LB (\partitions)$.
\end{proof}

\rev{Finally, the next proposition shows that any solution that is infeasible for Model~\eqref{eq:cclp-part} with~$\partitions$ has an objective value smaller than or equal to $\objval^{\LB}(\partitions)$ will never become feasible for Model~\eqref{eq:cclp-part} with~$\merged$.
\begin{proposition}
    \label{prop:cycling}
    Let~$\refined$ be a refinement of~$\partitions$ constructed by applying Algorithm~\ref{alg:refinement} and such that~$\min_{\newrefined\in\newsets^{\refined}} \singlecost_\newrefined > \objval^{\LB}(\partitions)$. Let~$\merged$ be a merger of~$\refined$ obtained by applying Algorithm~\ref{alg:merge}. If~$\lbsol$ is not feasible for Model~\eqref{eq:cclp-part} with~$\partitions$ and is such that~$\objfun (\lbsol) \leq \objval^{\LB}(\partitions)$ then $\lbsol$ is not feasible for Model~\eqref{eq:cclp-part} with~$\merged$.
\end{proposition}
\begin{proof}
    Since $\merged$ is a merger of~$\refined$, we have~$\card{\feasmerged(\lbsol)} \leq \card{\feasrefined(\lbsol)}$. Similarly, we have~$\card{\feasrefined(\lbsol)} \leq \card{\feaspartitions(\lbsol)}$, otherwise~$\min_{\newrefined\in\newsets^{\refined}} \singlecost_\newrefined \leq \objval^{\LB}(\partitions)$ holds and we have a contradiction. Since the point~$\lbsol$ is infeasible for Model~\eqref{eq:cclp-part} with~$\partitions$ the inequality~$\card{\feaspartitions (\lbsol)} \leq \card{\partitions} - \lfloor \tau\card{\scenarios} \rfloor - 1$ holds following Remark~\ref{remark:chance-cons-equi}. Further, Proposition~\ref{prop:merger-construction} shows that Algorithm~\ref{alg:merge} creates a merger of size~$\card{\merged} = \card{\partitions}$. Combining these equations results in the inequality
    \begin{align*}
        \card{\feasmerged(\lbsol)} &\leq \card{\feasrefined(\lbsol)}  \leq \card{\feaspartitions(\lbsol)},\\
                                   & \leq \card{\partitions} - \lfloor \tau\card{\scenarios} \rfloor - 1,\\
                                   &  = \card{\merged} - \lfloor \tau\card{\scenarios} \rfloor - 1,
    \end{align*}
    which, by Remark~\ref{remark:chance-cons-equi}, implies that~$\lbsol$ is not feasible for Model~\eqref{eq:cclp-part} with~$\merged$.
\end{proof}

Proposition~\ref{prop:cycling} describes the conditions in which a point remains infeasible after a refinement and merge operations. These conditions are always satisfied using Algorithms 2 and 3 for refinement and merge respectively and when merging is only performed when the condition $\min_{\newrefined\in\newsets^{\refined}} \singlecost_\newrefined > \objval^{\LB}(\partitions)$ is satisfied. In that case, a point excluded during previous refinement and merge operations cannot become feasible in a subsequent partition. As a consequence, our algorithm is ensured never to cycle.}

\rev{\subsection{Mergers with equal value.}
  \label{sec:suff-part-thro}

  In what follows, \revtwo{we present sufficient conditions} for the existence of a merger~$\merged$ of any partition~$\partitions$ that satisfies~$\objval^\LB (\merged) = \objval^\LB (\partitions)$. \revtwo{This section provides intuition on conditions in which the \APM is likely to terminate early. Indeed, by applying these results to the trivial partition~$\partitions = \{\{\scenario_1\}, \dots, \{\scenario_{\card{\scenarios}}\}\}$ we can identify when a small size partition of the original scenario set $\scenarios$ exists. Such results also pave the way for a deeper theoretical characterization of completely sufficient partitions in the sense of~\cite{song2015adaptive}, which lies beyond the scope of this paper.}

  Note that, contrarily to the previous section, the set~$\merged$ is now a merger of~$\partitions$ and not a merger of~$\refined$. We introduce two sets necessary for discussing the aforementioned results. First, for an arbitrary~$\partitions$, we define $\objregion (\partitions) = \{\decvar \in \decvarset : \ \objfun (\decvar) \leq \objval^\LB (\partitions) \}$ the set of points~$\decvar \in \decvarset$ that have an objective value smaller than or equal than~$\objval^\LB (\partitions)$. Furthermore, for an arbitrary partition $\partitions$ and any positive integer $\sizeparam \in \naturals$ we define the set of $\sizeparam$-feasible points as
\begin{equation*}
    \feasregion (\partitions, \sizeparam) = \left \{ \decvar \in \objregion (\partitions) : \sum_{\partition \in \partitions} \indvar_\partition (\decvar) \geq \card{\partitions} - \lfloor \threshold \card{\partitions} \rfloor - \sizeparam \right \}.
\end{equation*}
This set represents the region of $\objregion (\partitions)$ satisfying $\delta$ fewer subsets than necessary in Model~\eqref{eq:cclp-part} with~$\partitions$.

\begin{proposition}
    \label{prop:union-inter}
    Let $\partitionsleft,\partitionsright \subseteq \partitions$ such that $\card{\partitionsleft} = \card{\partitionsright} = \sizeparam \in \naturals_{\geq 1}$ and $\partitionsleft \cap \partitionsright = \emptyset$. If~$\feasregion (\partitions, \sizeparam) \subseteq (\cap_{\partition \in \partitionsleft} \probset^\partition) \cup (\cap_{\partition \in \partitionsright} \probset^\partition)$, then there exists a merger~$\merged$ of size $\card{\merged} = \card{\partitions} - \sizeparam $ such that
    \begin{equation*}
        \objval^\LB (\merged) = \objval^\LB (\partitions).
    \end{equation*}
\end{proposition}
\begin{proof}
    We construct~$\merged$ by removing all the subsets of~$\partitionsleft$ and~$\partitionsright$ from~$\partitions$ by creating~$\sizeparam$ new subsets, each obtained by taking the union of one subset in~$\partitionsleft$ and one subset in ~$\partitionsright$. The order in which the subsets are taken has no importance. $\merged$ is a merger of~$\partitions$ and~$\card{\merged} = \card{\partitions} - \sizeparam$.

    It remains to show that~$\objval^\LB (\merged) = \objval^\LB (\partitions)$. Assume that there exists a point~$\lbsol$ feasible for Model~\eqref{eq:cclp-part} with~$\merged$ and such that~$\objfun(\lbsol) < \objval^\LB (\partitions)$. First, we show the validity of two inequalities for~$\lbsol$. Second, we combine these inequalities to show that~$\lbsol$ is necessarily feasible for Model~\eqref{eq:cclp-part} with~$\partitions$, hence contradicting the statement. Recall the notation~$\newsets^\merged = \merged \setminus \partitions$.

    \textit{Inequality 1. } Let $\newmerged$ be an arbitrary subset in~$\newsets^\merged$. Let~$\partition_1 \in \partitions_1 $ and~$\partition_2 \in \partitions_2$ be the unique subsets selected to create~$\newmerged$, \ie~$\newmerged = \partition_1 \cup \partition_2$. If~$\indvar_\newmerged (\lbsol) = 1$ then, by construction,~$\indvar_{\partition_1} (\lbsol) = \indvar_{\partition_2} (\lbsol) = 1$.

    Recall that~$\lbsol$ is feasible for Model~\eqref{eq:cclp-part} with~$\merged$. This implies that~$\lbsol \in \feasregion (\partitions , \sizeparam ) \subseteq  (\cap_{\partition \in \partitionsleft} \probset^\partition) \cup (\cap_{\partition \in \partitionsright} \probset^\partition)$. Hence,~$\lbsol \in \cap_{\partition \in \partitionsright} \probset^\partition$ or~$\lbsol \in \cap_{\partition \in \partitionsleft} \probset^\partition$. As a consequence, if~$\indvar_\newmerged (\lbsol) = 0$ we have~$\indvar_{\partition_1} (\lbsol)  = 1$ or~$\indvar_{\partition_2} (\lbsol) = 1$. For every~$\newmerged \in \newsets^\merged$, the inequality
    \begin{equation*}
        \indvar_{\partition_1} (\lbsol) + \indvar_{\partition_2} (\lbsol)  \geq \indvar_\newmerged (\lbsol) + 1
    \end{equation*}
    is valid. Summing over all~$\newmerged \in \newsets^\merged$ yields
    \begin{equation}
        \label{ineq:res-set}
        \sum_{\partition_1 \in \partitionsleft}  \indvar_{\partition_1} (\lbsol) + \sum_{\partition_2 \in \partitionsright} \indvar_{\partition_2} (\lbsol) \geq  \sum_{\newmerged \in \newsets^\merged} \indvar_\newmerged (\lbsol) + \card{\newsets^\merged}.
    \end{equation}

    \textit{Inequality 2. } Let $\newmerged$ now be an arbitrary subset in~$\merged \setminus \newsets^\merged$. By construction, there is exactly one~$\partition \in \partitions$ such that~$\partition = \newmerged$. This means that for every~$\newmerged \in \merged  \setminus \newsets^\merged$, the inequality
    \begin{equation*}
        \indvar_{\partition} (\lbsol)  \geq \indvar_\newmerged (\lbsol)
    \end{equation*}
    is valid. Summing over all~$\newmerged \in \merged  \setminus \newsets^\merged$ yields
    \begin{equation}
        \label{ineq:orig-set}
        \sum_{\partition \in \partitions \setminus \partitionsleft \setminus \partitionsright}  \indvar_{\partition} (\lbsol) \geq  \sum_{\newmerged \in \merged  \setminus \newsets^\merged} \indvar_\newmerged (\lbsol).
    \end{equation}
    Finally, by combining Inequality~\eqref{ineq:res-set} and Inequality~\eqref{ineq:orig-set}, we obtain
    \begin{align*}
        \sum_{\partition \in \partitions}  \indvar_{\partition} (\lbsol)
        &= \sum_{\partition_1 \in \partitionsleft}  \indvar_{\partition_1} (\lbsol) + \sum_{\partition_2 \in \partitionsright} \indvar_{\partition_2} (\lbsol) + \sum_{\partition \in \partitions \setminus \partitionsleft \setminus \partitionsright}  \indvar_{\partition} (\lbsol),\\
        &\geq \sum_{\newmerged \in \merged} \indvar_\newmerged (\lbsol) + \card{\newsets^\merged},\\
        &\geq \card{\partitions} - \lfloor \tau\card{\scenarios} \rfloor,
    \end{align*}
    where, in order, we use Inequality~\eqref{eq:condition-removal} and~$\card{\partitions} = \card{\merged} + \card{\newsets^\merged}$. The last inequality states that the point~$\lbsol$ is feasible for Model~\eqref{eq:cclp-part} with~$\partitions$, which is a contradiction.
\end{proof}

% Proposition~\ref{prop:union-inter} states a condition for merging a partition while preserving its objective value. It is based on a coverage conditions on the feasible region of the subsets.

\begin{corollary}
    \label{cor:less-hard}
    Let $\subpartitions \subseteq \partitions$ such that $\card{\subpartitions} = \sizeparam \in \naturals_{\geq 2}$. If~$\feasregion (\partitions, \sizeparam - 1) \subseteq \cap_{\partition \in \subpartitions} \probset^\partition$, then there exists a merger~$\merged$ of size $\card{\merged} = \card{\partitions} - \sizeparam + 1$ such that
    \begin{equation*}
        \objval^\LB (\merged) = \objval^\LB (\partitions).
    \end{equation*}
\end{corollary}
\begin{proof}
    Suppose without loss of generality that there is a numbering of the subsets of $\subpartitions$, \ie,~$\subpartitions = \{\partition_0, \dots, \partition_{\sizeparam - 1}\}$. We construct inductively a series of mergers~$\merged_\mergeind$. Let $\newmerged_1 = \partition_0 \cup \partition_1$, and~$\newmerged_\mergeind = \newmerged_{\mergeind - 1} \cup \partition_\mergeind$ for all $\mergeind \in \{1,\dots,\sizeparam - 1\}$. Then, we create $\merged_0 = \partitions \setminus \{\partition_0,  \partition_1\} \cup \{\newmerged_0\}$, and $\merged_\mergeind = \partitions \setminus \{\newmerged_{\mergeind - 1}\} \cup \{\newmerged_{\mergeind }\}$ for all $\mergeind \in \{1,\dots,\sizeparam - 1\}$. The set $\merged_{\sizeparam - 1}$ has size $\card{\merged} = \card{\partitions} - \sizeparam + 1$ and is a merger of $\partitions$ by construction. Since the property $\feasregion (\merged_\mergeind, \mergeind) \subseteq \feasregion (\partitions, \sizeparam - 1) \subseteq \cap_{\partition \in \subpartitions} \probset^\partition \subseteq (\probset^{\newmerged_{\mergeind - 1} } \cup \probset^{\partition_\mergeind})$ is satisfied for all $\mergeind \in \{1,\dots,\sizeparam - 1\}$, we can apply Proposition~\ref{prop:union-inter} in an inductive fashion and obtain $\objval^\LB (\merged_{\sizeparam - 1}) = \dots = \objval^\LB (\merged_{0}) = \objval^\LB (\partitions)$.
\end{proof}

% Corollary~\ref{cor:less-hard} states another condition for merging a partition while preserving its objective value. It shows that, if the intersection of a subset of~$\partitions$ with size~$\sizeparam$ spans the~$(\sizeparam -1) $-feasible region of~$\partitions$, then a merger can be constructed without changing the objective value. This condition will likely be satisfied when the regions $\probset^\partition$ are similar for most of the $\partition \in \partitions$.

Proposition~\ref{prop:union-inter} and Corollary~\ref{cor:less-hard} show that merging operations can be conducted without decreasing the objective value when some subsets have overlapping feasible regions. Intuitively, if a subset is such that its feasible region strongly overlaps with the feasible region of a CCSP, most of the optimal points of that CCSP are likely to lay inside of the feasible region of the overlapping subset. Hence, merging this subset with other similar subsets will not create new feasible points that induce a lower objective value.
% Proposition~\ref{prop:union-inter} and Corollary~\ref{cor:less-hard} formalize this intuition.

% Proposition~\ref{prop:union-inter} and Corollary~\ref{cor:less-hard} require the description of the set~$\feasregion (\partitions, \sizeparam)$. Computing this set is hard in practice. However, they provide intuition about the conditions under which small mergers are likely to exist. We can expect an APM to work well when there is significant overlap between the scenarios or when their feasible regions are similar.

\revtwo{Proposition~\ref{prop:union-inter} and Corollary~\ref{cor:less-hard} provide a theoretical understanding of when small-size partitions are likely to exist.} An interesting extension \revtwo{would be to} study the existence of completely sufficient partitions, as introduced in~\cite{song2015adaptive}. Indeed, a merger~$\merged$ of the trivial partition~$\partitions = \{\{\scenario_1\}, \dots, \{\scenario_{\card{\scenarios}}\}\}$ is a completely sufficient partition in the sense of~\cite{song2015adaptive}, if~$\objval^\LB (\merged) = \objval^\LB (\partitions) = v^*$.
\revtwo{Proposition~\ref{prop:union-inter} and Corollary~\ref{cor:less-hard} may not provide a numerical advantage since the set~$\feasregion (\partitions, \sizeparam)$ cannot be easily described in practice. Therefore, these properties are not used in the numerical experiments presented in Section~\ref{sec:numerical-results}.}
}

%%% Local Variables:
%%% mode: latex
%%% TeX-master: "../main"
%%% End:

\section{Strong Partitions and Practical Strategies}
\label{sec:strong-part}

This section covers the remaining components that complement the presentation of the general \APM summarized in Algorithm~\ref{alg:adaptive-partitioning}. \rev{First, we highlight how tight big-M coefficients can be computed by taking advantage of Proposition~\ref{prop:lb-partition}.} Second, we discuss how to build partitions from scratch such that solving Model~\eqref{eq:cclp-part} yields tight lower bounds on~$\objval^*$. Third, we propose a method for refining selected subsets by maximizing the value of~$\min_{\newrefined\in\newsets^\refined} \singlecost_\newrefined$. Fourth, we explain how we select scenarios to be refined and merged to take advantage of the proposed refinement method. Finally, a heuristic for projecting an infeasible solution to the set of feasible solutions of the original model~\eqref{eq:cclp-reform} is proposed.

\subsection{Big-M Tightening}
\label{sec:big-m-ideas}
As discussed in Section~\ref{sec:introduction}, tightening big-M coefficients is of major importance for problems with indicator variables \cite{Belotti2016handling}. The value of big-M coefficients directly influences the size of the continuous relaxation of Model~\eqref{eq:cclp-reform-bigm}. Thus, it also indirectly influences the time required to solve Model~\eqref{eq:cclp-reform-bigm} to optimality. For any chance-constrained problem, the tightest value for the big-$\bigM$ parameter of the $\entry$-th constraint of a scenario $\scenario$ is obtained by solving the following optimization problem, see, \eg,~\cite{song2014chance},
\begin{subequations}
    \label{eq:exact-bigm}
    \begin{align}
        \bestbigM^\scenario_\entry = \max_{\decvar \in \decvarset} \quad
        & \ineqsys^\scenario_\entry(\decvar) \\
        \st \quad
        & \ineqsys^\scenario(\decvar)  \leq
          \bigM^\scenario (1-\indvar^\scenario), \quad \scenario \in \scenarios, \label{eq:indicator-var-bigm-comp}\\
        & \sum_{\scenario \in \scenarios} \proba_\scenario\indvar_\scenario \geq 1-\threshold, \label{eq:chance-cons-bigm-comp}\\
        & \indvar_\scenario \in \{0,1\}, \quad \scenario \in \scenarios, \label{eq:scenario-def-bigm-comp}
    \end{align}
\end{subequations}
Model~\eqref{eq:exact-bigm} is a chance-constrained problem and it can be as hard to solve as
Model~\eqref{eq:cclp-reform-bigm}. Thus, it is largely unpractical to consider such an approach for every big-$\bigM$ that needs to be computed.

The partitioned CCSP~\eqref{eq:cclp-part} can be used to obtain tight big-M coefficients. Indeed, Proposition~\ref{prop:refined-relax} states that Model~\eqref{eq:cclp-part} is a relaxation of Model~\eqref{eq:cclp-reform} which means that Model~\eqref{eq:cclp-part} can be used to find a tight approximation of $\bestbigM^\scenario_\entry$. Given a partition~$\partitions$, we state the partitioned big-M tightening model,
\begin{subequations}
    \label{eq:bigm-part}
    \begin{align}
        \bigM^\scenario_\entry = \max_{\decvar \in \decvarset} \quad
        & \ineqsys^\scenario_\entry(\decvar) \\
        \st \quad
        & \ineqsys^\scenario(\decvar)  \leq \bigM^\scenario (1-\indvar^\partition), \quad \partition\in\partitions,\scenario \in \partition, \label{eq:indicator-var-bigm-part}\\
        & \sum_{\partition \in \partitions} \proba_\partition\indvar_\partition \geq \sum_{\partition\in\partitions}\proba_\partition-\threshold, \label{eq:chance-cons-bigm-part}\\
        & \indvar_\partition \in \{0,1\}, \quad \partition \in \partitions, \label{eq:scenario-def-bigm-part}
    \end{align}
\end{subequations}
where~$\bestbigM^\scenario_\entry \leq \bigM^\scenario_\entry$. We observe that any partition~$\partitions$ of $\scenarios$ allows to produce a valid big-$\bigM$ value.

\subsection{Partitions for Tight Lower Bounds}
\label{sec:init}

In what follows, we use the minimum subset cost coefficients~$\singlecost_\partition$ introduced in Equation~\eqref{eq:single-cost-definition}. Moreover, we use the coefficients~$\singlecost_\scenario$, which are defined exactly as in Equation~\eqref{eq:single-cost-definition} but the set~$\probset^\partition$ is replaced by the set~$\probset^\scenario$.

We describe a procedure for building a partition~$\partitions$ of size~$\lfloor \threshold \abs{\scenarios} \rfloor + 1$ such that~$\objval^\LB (\partitions) \geq \objval_\quantile $, where~$\objval_\quantile$ is the well-known quantile bound~\cite{xie2018quantile}. Let~$\permutation$ be a permutation of~$\scenarios$ that satisfies~$\singlecost_{\permutation_1} \geq \dots \geq \singlecost_{\permutation_{\card{\scenarios}}}$.  In~\cite{ahmed2017nonanticipative} the quantile bound is defined as~$\objval_{\quantile} = \singlecost_{\permutation_q}$ where~$q = \min\{k \in \{1,\dots,\card{S}\} : \sum_{\permindex = 1}^{k} \proba_{\permutation_\permindex} > \threshold\}$.
\begin{proposition}[from~\cite{ahmed2017nonanticipative}]
    \label{prop:quantile-bound}
    The quantile bound~$\objval_\quantile$ is a lower bound on the optimal objective of the CCSP~\eqref{eq:cclp-reform}, \ie,~$\objval_\quantile\leq\objval^*$.
\end{proposition}

We now discuss how to obtain a partition~$\partitions$ such that~$\objval^\LB (\partitions) \geq \objval_\quantile$. We use the permutation~$\permutation$ introduced earlier. First, we fix the size of~$\partitions$ to~$\lfloor \threshold \abs{\scenarios} \rfloor + 1$. Second, we assign the scenarios~$\permutation_i$ to the \rev{set~$\partition_j$, where $j \equiv i \mod \card{\partitions}$} for all~$i\in\{1,\dots,\card{\scenarios}\}$. In other words, scenario~$\permutation_1$ is assigned to set~$\partition_1$, scenario~$\permutation_2$ is assigned to set~$\partition_2$, and so forth, until the set~$\partition_{\card{\partitions}}$ is reached. This assignment is continued at~$\partition_1$ until all scenarios are assigned to a set.
\begin{proposition}
    \label{prop:stronger-initial}
    If~$\partitions$ is created by a sequential assignment of the ordering given by the permutation~$\permutation$ into~$\lfloor \threshold \abs{\scenarios} \rfloor + 1$ disjoint sets, then
    \begin{equation*}
         \objval^\LB(\partitions) \geq \objval_\quantile.
    \end{equation*}
\end{proposition}
\begin{proof}
    From Remark~\ref{remark:chance-cons-equi} it follows that
    \begin{equation*}
        \sum_{\partition\in\partitions} \indvar_\partition  \geq
        \abs{\partitions} -\lfloor  \threshold \abs{\scenarios}  \rfloor = 1.
    \end{equation*}
    Since we sequentially dispatch the scenarios based on~$\singlecost_\scenario$, each set~$\partition\in\partitions$ will contain at least one scenario~$\permutation_\permindex$ with~$\permindex \in \{1, \dots, \lfloor \threshold \abs{\scenarios}  \rfloor + 1 \}$. When Assumption~\ref{ass:equiprobable} holds, we know that~$q = \lfloor \threshold \abs{\scenarios}  \rfloor + 1$ and it follows that
    \begin{equation*}
        \objval^\LB(\partitions) \geq \min_{\permindex \in \{1, \dots, \lfloor \threshold \abs{\scenarios}  \rfloor + 1 \}} \singlecost_{\permutation_l} = \singlecost_{\permutation_q} = \objval_\quantile.
    \end{equation*}
\end{proof}
Proposition~\ref{prop:stronger-initial} implies that if~$\partitions$ is obtained by a sequential assignment of the permutation~$\permutation$ into~$\lfloor \threshold \abs{\scenarios} \rfloor + 1$ disjoint sets, then~$\objval^\LB(\partitions) \geq \objval_\quantile$. This result, when combined with Theorems~\ref{thm:strictly-incr-lower} and~\ref{thm:strict-increase-merge}, ensures a strict increase of a tight lower bound produced by Algorithm~\ref{alg:adaptive-partitioning}, while keeping the size of the considered partitions as small as possible.

\subsection{Refinements that Promote Merging}
\label{sec:accurate-split}
We now propose a model that, when solved to optimality, splits the subset~$\newrefined_1$ into the subsets~$\newrefinedleft$ and~$\newrefinedright$ as described in Algorithm~\ref{alg:refinement}. We propose to split~$\newrefined_1$ in a way that maximizes the value of~$\min_{\newrefined \in \{\newrefinedleft, \newrefinedright\}} \singlecost_{\newrefined}$. Two observations support this approach. First, as described in Proposition~\ref{prop:refinement-bound}, when the subsets~$\newrefinedleft$ and~$\newrefinedright$ are such that~$\min_{\newrefined \in \{\newrefinedleft, \newrefinedright\}} \singlecost_{\newrefined} > \objval^{\rev{\LB}} (\partitions)$ we know that a merger~$\merged$ of size~$\card{\partitions}$ can be constructed so that~$\objval^\LB(\merged) > \objval^\LB(\partitions)$. Second, it is preferable that the value of~$\min_{\newrefined \in \{\newrefinedleft, \newrefinedright\}} \singlecost_{\newrefined}$ is as large as possible because the feasible regions~$\probset^{\newrefinedleft}$ and~$\probset^{\newrefinedright}$ may be selected in the optimal solution of Model~\eqref{eq:cclp-part} with~$\refined$. Indeed, if~$\indvar_{\newrefinedleft}$ or~$\indvar_{\newrefinedright}$ are equal to one for the optimal solution of Model~\eqref{eq:cclp-part} with~$\refined$, then~$\objval^{\rev{\LB}} (\refined) > \min(\singlecost_{\newrefinedleft},\singlecost_{\newrefinedright})$.

A naive way to maximize~$\min_{\newrefined \in \{\newrefinedleft, \newrefinedright\}} \singlecost_{\newrefined}$ is computing its value for every possible split of~$\newrefined_1$. The number of models that need to be solved is given by the Stirling numbers of the second kind~\cite{graham1989concrete}. Hence, the complexity of this naive approach grows exponentially as the number of scenarios inside the subset~$\newrefined_1$ increases.

We propose an optimization-based approach for chance-constrained linear problems that can be used when the decision variables are continuous or binary. Our approach is based on formulating the splitting problem as a bi-level optimization problem~\cite{dempe2002foundations}. The upper-level yields an assignment of scenarios to the subsets~$\newrefinedleft$ and~$\newrefinedright$ while maximizing the value of~$\min_{\newrefined \in \{\newrefinedleft, \newrefinedright\}} \singlecost_{\newrefined}$. The lower-level computes the subset costs~$\singlecost_{\newrefinedleft} $ and~$\singlecost_{\newrefinedright}$ given a scenario assignment. For every~$\scenario \in \newrefined_1 $ we introduce two binary assignment variables~$\linkvar_{\scenario \newrefinedleft}$ and~$\linkvar_{\scenario \newrefinedright}$ to track whether this scenario is assigned to subset~$\newrefinedleft$ or~$\newrefinedright$. The upper-level problem is given by
\begin{subequations}
  \label{eq:accurate-objective-increase}
  \begin{align}
      \singlecost_{\text{div}} = \max_{\linkvar} \quad
    & \min_{\newrefined \in \{\newrefinedleft, \newrefinedright\}} \singlecost_{\newrefined} (\linkvar) \\
    \st \quad
    & \sum_{\newrefined \in \{\rev{\newrefinedleft}, \newrefinedright\}} \linkvar_{ \scenario \newrefined} = 1, \quad \scenario \in \rev{\newrefined_1}, \label{eq:forced-assignement}\\
    & \sum_{\scenario \in \newrefined_1 \cap \infscenarios(\lbsol)} \linkvar_{\scenario \newrefined} \geq 1, \quad \newrefined \in \{\newrefinedleft, \newrefinedright\}, \label{eq:forced-inf-assignement}\\
      & \rev{\linkvar_{ \scenario \newrefined} \in \{0,1\}, \quad \scenario \in \newrefined_1, \, \newrefined \in \{\newrefinedleft, \newrefinedright\}}.
  \end{align}
\end{subequations}
Constraint~\eqref{eq:forced-assignement} states that every scenario is assigned to a subset. Constraint~\eqref{eq:forced-inf-assignement} ensures that at least one~$\scenario \in \infscenarios (\lbsol)$ is assigned to each subset.

\rev{When the chance-constrained linear problem involves only continuous variables,} the value of~$\singlecost_{\newrefined}(\linkvar)$ for~$\newrefined \in \{ \newrefinedleft, \newrefinedright\}$ is the solution of the lower-level problem,
\begin{subequations}
    \label{eq:subset-objective-increase}
    \begin{align}
        \singlecost_{\newrefined}(\linkvar) = \min_{\decvar} \quad
                  & \cost^\top \decvar\\
        \st \quad &  \varmat^\decvarset \decvar \geq \constmat^\decvarset, \label{eq:decvarset-lp} \\
                  & \linkvar_{\scenario \newrefined}\varmat^\scenario \decvar \geq \linkvar_{\scenario \newrefined}\constmat^\scenario, \quad \scenario\in\newrefined_1. \label{eq:asset-lp}
    \end{align}
\end{subequations}
Model~\eqref{eq:subset-objective-increase} computes the single subset cost~$\singlecost_{\newrefined}$ \rev{as defined in Equation~\eqref{eq:single-cost-definition}} for a set of assignment variables~$\linkvar$. Constraint~\eqref{eq:decvarset-lp} is the equivalent of~$\decvar\in\decvarset$ when we assume that Model~\eqref{eq:cclp-reform-bigm} is linear. Constraint~\eqref{eq:asset-lp} is the equivalent of~$\decvar \in \probset^\scenario$ when we assume that Model~\eqref{eq:cclp-reform-bigm} is linear. \rev{If integer variables are present, Model~\eqref{eq:subset-objective-increase} is a linear relaxation and provides a lower bound on the true value of the single subset cost.}
\begin{proposition}
    \label{prop:single_level_acc_obj}
    Model~\eqref{eq:accurate-objective-increase} can be reformulated as a single-level problem of the form
  \begin{subequations}
    \label{eq:accurate-objective-incrgease-single}
    \begin{align}
      \singlecost_{\text{div}} = \max_{\linkvar,\dualvar} \quad
      & \intervar \\
      \st \quad
      &  \intervar \leq (\constmat^{\decvarset})^\top\dualvar^{\decvarset\newrefined}
        + \sum_{\scenario\in \newrefined_1}
        (\constmat^\scenario)^\top\dualvar^{\scenario\newrefined}, \quad
        \newrefined\in\{\newrefinedleft,\newrefinedright\},\\
      & (\varmat^\decvarset)^\top \dualvar^{\decvarset\newrefined}  +
        \sum_{\scenario\in\newrefined_1}
        (\varmat^\scenario)^\top
        \dualvar^{\scenario\newrefined}= \cost, \quad
        \newrefined\in\{\newrefinedleft,\newrefinedright\},\\
      & 1 - \linkvar_{\scenario\newrefined} = \boldone (\dualvar^{\scenario\newrefined}_\consindex
        = 0), \quad \scenario \in \newrefined_1,\consindex\in\consindexset(\scenario),\newrefined\in\{\newrefinedleft,\newrefinedright\},\\
      &\sum_{\newrefined \in \{\newrefinedleft, \newrefinedright\}} \linkvar_{ \scenario \newrefined} = 1, \quad \scenario \in
        \newrefined_1,\\
      & \sum_{\scenario \in \newrefined_1 \cap \infscenarios(\lbsol)} \linkvar_{\scenario \newrefined} \geq 1, \quad \newrefined \in \{\newrefinedleft, \newrefinedright\},\\
      & \rev{\linkvar_{ \scenario \newrefined} \in \{0,1\}, \quad \scenario \in \newrefined_1, \, \newrefined \in \{\newrefinedleft, \newrefinedright\}},\\
      & \dualvar\geq0.
    \end{align}
  \end{subequations}
\end{proposition}

\begin{proof}
    Since Model~\eqref{eq:subset-objective-increase} is linear, strong duality holds. Let~$\dualvar$ be the dual variables of Model~\eqref{eq:subset-objective-increase}. The dual of Model~\eqref{eq:subset-objective-increase} for~$\newrefined \in \{ \newrefinedleft, \newrefinedright\}$ reads
    \begin{subequations}
      \label{eq:dual-subset-objective-increase}
      \begin{align}
          \singlecost_{\newrefined}(\linkvar) = \max_{\dualvar} \quad
        & (\constmat^\decvarset)^\top\dualvar^\decvarset + \sum_{\scenario\in
          \rev{\newrefined_1}}
          \linkvar_{\scenario \newrefined}(\constmat^\scenario)^\top\dualvar^\scenario\\
        \st \quad
        &  (\varmat^\decvarset)^\top \dualvar^\decvarset +
          \sum_{\scenario\in
          \rev{\newrefined_1}}
          \linkvar_{\scenario \newrefined} (\varmat^\scenario)^\top
          \dualvar^\scenario= \cost,\\
        & \dualvar \geq 0.
      \end{align}
    \end{subequations}
    \rev{Let~$\consindexset(\scenario)$ be the set of constraint indices for scenario~$\scenario\in\scenarios$. For every~$\consindex \in \consindexset(\scenario)$, the dual variable~$\dualvar^\scenario_\consindex$ in Model~\eqref{eq:dual-subset-objective-increase} is always multiplied by the same binary variable~$\linkvar_{\scenario \newrefined}$. The product $\linkvar_{\scenario \newrefined} \cdot \dualvar^\scenario_\consindex$ is equal to~$\dualvar^\scenario_\consindex$ when~$\linkvar_{\scenario \newrefined} = 1$ and $0$ when~$\linkvar_{\scenario \newrefined} = 0$. Hence, we can replace~$\linkvar_{\scenario \newrefined} \cdot \dualvar^\scenario_\consindex$ with~$\dualvar^\scenario_\consindex$ and add the constraint~$1 - \linkvar_{\scenario\newrefined} = \boldone (\dualvar^\scenario_\consindex = 0)$. Model~\eqref{eq:dual-subset-objective-increase} now reads}
    \begin{subequations}
      \label{eq:dual-subset-objective-increase-reform}
      \begin{align}
      \singlecost_{\newrefined}(\linkvar) = \max_{\dualvar} \quad
        & (\constmat^\decvarset)^\top\dualvar^\decvarset
          + \sum_{\scenario\in \rev{\newrefined_1}}
          (\constmat^\scenario)^\top\dualvar^\scenario \\
        \st \quad
        &  (\varmat^\decvarset)^\top \dualvar^\decvarset +
          \sum_{\scenario\in \rev{\newrefined_1}}
           (\varmat^\scenario)^\top
          \dualvar^\scenario= \cost,\\
        & 1 - \linkvar_{\scenario\rev{\newrefined}} = \boldone (\dualvar^\scenario_\consindex
          = 0), \quad \scenario \in \rev{\newrefined_1},\consindex\in\consindexset(\scenario),\\
        & \dualvar \geq 0.
      \end{align}
    \end{subequations}

    Further, we introduce the variable~$\intervar = \min_{\newrefined \in \{\newrefinedleft, \newrefinedright\}} \singlecost_{\newrefined} (\linkvar)$. By construction,~$\intervar\leq\singlecost_{\newrefined}(\linkvar)$ holds for all~$\newrefined \in \{\newrefinedleft, \newrefinedright\}$. Hence, Model~\eqref{eq:accurate-objective-increase} is equivalent to
    \begin{subequations}
      \label{eq:accurate-objective-increase-reform}
      \begin{align}
          \singlecost_{\text{best}} = \max_{\linkvar} \quad
        & \intervar \\
        \st \quad
        & \intervar \leq \singlecost_{\newrefined}(\linkvar), \quad \newrefined \in \{\newrefinedleft, \newrefinedright\},\\
        & \sum_{\newrefined \in \{\newrefinedleft, \newrefinedright\}} \linkvar_{ \scenario \newrefined} = 1, \quad \scenario \in
          \rev{\newrefined_1},\\
        & \sum_{\scenario \in \newrefined_1 \cap \infscenarios(\lbsol)} \linkvar_{\scenario \newrefined} \geq 1, \quad \newrefined \in \{\newrefinedleft, \newrefinedright\},\\
      & \rev{\linkvar_{ \scenario \newrefined} \in \{0,1\}, \quad \scenario \in \newrefined_1, \, \newrefined \in \{\newrefinedleft, \newrefinedright\}}.
      \end{align}
    \end{subequations}
    Finally, we replace~$\singlecost_{\newrefined}(\linkvar)$ in Model~\eqref{eq:accurate-objective-increase-reform} by Model~\eqref{eq:dual-subset-objective-increase-reform}, which yields Model~\eqref{eq:accurate-objective-incrgease-single}.
\end{proof}
\rev{Proposition~\ref{prop:single_level_acc_obj} provides an efficient method to compute refinements that maximize the value of~$\min_{\newrefined \in \{\newrefinedleft, \newrefinedright\}} \singlecost_{\newrefined}$. The solution of Model~\eqref{eq:accurate-objective-incrgease-single} is such that $\min_{\newrefined \in \{\newrefinedleft, \newrefinedright\}} \singlecost_{\newrefined}(\linkvar^*) = \singlecost_{\text{div}}$ when all variables are continuous and $\min_{\newrefined \in \{\newrefinedleft, \newrefinedright\}} \singlecost_{\newrefined} (\linkvar^*) \geq \singlecost_{\text{div}}$ when integer variables are present. In both cases, if the condition~$\singlecost_{\text{div}} > \objval (\partitions)$ is satisfied, then a merger~$\merged$ of size~$\card{\partitions}$ that ensures~$\objval^\LB(\merged) > \objval^\LB(\partitions)$ can be constructed.}

\begin{remark}
    \rev{It is straightforward to generalize} the single-level reformulation of Model~\eqref{eq:accurate-objective-incrgease-single} to the case where a subset is split into more than two subsets.
  \end{remark}

\subsection{Subset Selection for Refinement and Merging}
\label{sec:subs-select-refin}

We now explain how to select the subsets used in refinement and merging operations at each iteration of our adaptive method. Our methods aim to maximize the number of merging operations that are carried out.

We know from Proposition~\ref{prop:refinement-bound} that any refinement such that~$\min_{\newrefined \in \{\newrefinedleft, \newrefinedright\}} \singlecost_{\newrefined} > \objval^{\rev{\LB}} (\partitions)$ allows to perform a subsequent merge operation. Hence, when possible, we always select a subset~$\newrefined_1\in \infrefined (\lbsol)$ that satisfies~$\singlecost_{\divided} > \objval^\LB (\partitions)$ where~$\singlecost_{\divided}$ is obtained by solving Model~\eqref{eq:accurate-objective-incrgease-single}. When several subsets satisfy the condition~$\singlecost_{\divided} > \objval^\LB (\partitions)$, we select the one with the smallest~$\singlecost_{\divided}$ value. When no subset satisfies the condition, \ie, all subsets are such that~$\singlecost_{\divided} \le \objval^\LB (\partitions)$, we select~$\newrefined_1 = \argmax_{\newrefined \in \infrefined (\lbsol)} \singlecost_{\divided}$.

When merging, we select the feasible subsets~$\newmerged \in \feasmerged(\lbsol)$ with the largest single subset costs. When an infeasible subset is necessary to construct the merger (see the second case in Algorithm~\ref{alg:merge}), we select the subset with the largest value for~$\singlecost_\scenario$. This strategy ensures that~$\objval^\LB(\merged)$ is large when the newly created subset~$\newmerged_{\extrasize+2}$ is selected in the next iteration.

\subsection{Recovering Feasible Solutions}
\label{sec:projection}
We propose a simple projection heuristic to recover valid primal solutions when the optimal solution~$\lbsol$ of Model~\eqref{eq:cclp-part} is not feasible for Model~\eqref{eq:cclp-reform}. We construct this feasible point~$\ubsol$ by selecting additional scenarios to be satisfied until Constraint~\eqref{eq:chance-cons2} is valid. We introduce the set~$\projset$ which contains all the scenarios that are satisfied for obtaining the point~$\ubsol$. Initially, this set~$\projset$ is composed of every~$\scenario \in \feasscenarios (\lbsol)$. Then, scenarios~$\scenario \in \infscenarios(\lbsol)$ are greedily added to~$\projset$ based on their feasibility \wrt~$\lbsol$. That is, we iteratively add the scenario with the smallest value of~$\max_{\consindex \in \consindexset (\scenario)} \ineqsys^\scenario_\consindex(\lbsol)$ to~$\projset$ until~$\card{\projset} = \card{\scenarios} - \lfloor \threshold \card{\scenarios} \rfloor$. Finally, we obtain~$\ubsol$ by solving the following model

\begin{subequations}
  \begin{align*}
      \objval_{\text{proj}} = \min_{\decvar \in \decvarset} \quad
    & \objfun(\decvar) \\
    \st \quad
    & \ineqsys^\scenario(\decvar) \leq 0,
      \quad \scenario \in \projset.
  \end{align*}
\end{subequations}

Further, each time the value of~$\singlecost$ is computed for a scenario or a subset we check if the associated solution is feasible for Model~\eqref{eq:cclp-reform-bigm} and improves the current best upper bound.

%%% Local Variables:
%%% mode: latex
%%% TeX-master: "../main"
%%% End:

\section{Numerical Study}
\label{sec:numerical-results}

This section presents the numerical results of the \APM when compared to state-of-the-art methods. To assess the value of the \APM and its components, we run repeated experiments on CCSP instances taken from the literature. Our numerical study investigates the performance of our method by measuring the time taken to solve an instance to optimality, or the optimality gap when the instance cannot be solved to optimality in the allocated time. To understand the strengths and weaknesses of our method, we also study a single run of the \APM in detail. This allows us to analyze the occurrence and the effectiveness of the strategies proposed in Section~\ref{sec:strong-part}. For instance, we show how the lower and upper bounds evolve and how the size of the partition evolves as the number of iterations increases.

All computations have been executed on a remote server. Each experiment is run on a single core of an Intel Gold 6148 Skylake with \SI{2.4}{\giga\hertz} and is allocated \SI{16}{\giga\byte}~RAM. A time limit of \SI{120}{\minute} is enforced. Our implementation is made with the programming language \Python. All optimization models are solved using \Gurobi~10.0.3. The code used to produce all numerical results is publicly available at the online repository: \url{https://github.com/alexforel/AdaptiveCC}.

\subsection{Experimental Setting}
We follow the experimental setting of recent works on CCSPs in which multi-dimensional knapsack problems with either binary or continuous variables are used. All instances are generated according to the method described in~\cite{song2014chance} and \cite{ahmed2017nonanticipative}. Each chance-constrained instance is created by sampling a set of scenarios from a deterministic instance. A scenario is obtained by perturbing the left-hand side constraint matrix of the original deterministic instance. We consider three base instances:~\textsf{mk-10-10},~\textsf{mk-20-10}, and~\textsf{mk-40-30}, which have~10,~10, and~30 constraints per scenario, respectively, as well as~10,~20, and~40 decision variables, respectively. We generate five perturbed instances per deterministic instance.

\subsubsection{Implementation and Benchmarks}
The final adaptive method, referred to as~$\Part_{\final}$, is implemented following the description of the \APM displayed in Algorithm~\ref{alg:adaptive-partitioning}. Moreover, we implemented all the strategies discussed in Section~\ref{sec:strong-part}. Obtaining big-M coefficients using Model~\eqref{eq:bigm-part} in each iteration of the \APM is too time-consuming. Therefore, we use the less computationally heavy approach of \cite{Belotti2016handling} each time a new model is solved. In preliminary experiments, we observed that a large amount of big-M constraints~\eqref{eq:chance-cons-part} are unnecessary for solving a reduced model~\eqref{eq:cclp-part}. This is explained by the fact that each individual scenario is obtained by perturbing a deterministic set of constraints. We observe that, when scenarios are inside the same subset of a partition, a constraint of a scenario is often dominated by a constraint of another scenario. These dominated constraints are thus not necessary for representing the feasible set of the reduced model~\eqref{eq:cclp-part}. As a consequence, the computational performance of the \APM improves when \rev{implementing the big-M scenario constraints~\eqref{eq:chance-cons-part} as lazy constraints. Lazy constraints are initially inactive and placed in a lazy constraint pool. They are activated when a feasible solution is found that violates them, causing the solution to be discarded and the violated constraints to be added to the model. This is implemented by setting \Gurobi's \texttt{Lazy} parameter to~1 for all big-M constraints.}

We compare the \APM with \rev{two} benchmarks \rev{that} the big-M formulation of the original CCSP as given in Model~\eqref{eq:cclp-reform-bigm}. These two methods differ in the way the big-M parameters are computed. The first benchmark, referred to as \enquote{\texttt{Song Big-M}}, follows the method proposed in~\cite{song2014chance} and also evaluated in~\cite{ahmed2017nonanticipative}. This method is tailored to chance-constrained packing problems and therefore to the multi-dimensional chance-constrained knapsack problems that we consider. It is based on solving a series of single-dimensional continuous knapsack problems from which extremely tight upper bounds on the big-M coefficients can be computed using a quantile argument. The second benchmark, referred to as \enquote{\texttt{Belotti Big-M}}, is based on the problem-agnostic method of \cite{Belotti2016handling}. To obtain the big-M coefficients, a single-dimensional knapsack problem is solved for each scenario/constraint combination using a valid lower bound on the original CCSP. For both benchmarks, we do not see a distinct improvement when \Gurobi's \texttt{Lazy} parameter is set to 1 for big-M constraints~\eqref{eq:indicator-var-bigm}, and therefore leave it to its default value. To allow for a fair comparison between the results of different methods we include the time needed to compute the big-M coefficients in the total computation time.

\subsubsection{Big-M Computation Time}
Obtaining tight big-M coefficients can require a large amount of time. In particular, the method outlined in \cite{song2014chance} requires solving~$\card{\scenarios}^2 \card{\consindexset}^2$ single-dimensional continuous knapsack problems and thus scales quadratically with the number of constraints and scenarios. Our implementation of \cite{song2014chance} is coded in C\texttt{++} and interfaced with our code using Cython \cite{behnel2011cython}. Further, it exploits the symmetry of the problem to avoid performing unnecessary sorting operations. The time needed to obtain the big-M coefficients for our two benchmarks is given in Table~\ref{table:big_m_times}.

\begin{table}[!ht]
    \caption{Computation time needed to obtain all big-M parameters.}
    \label{table:big_m_times}
    \centerline{
    \begin{tabular}{c*{4}{r}}
        \toprule
        & & \multirow{2}{*}{\texttt{Song Big-M}} & \multicolumn{2}{c}{\texttt{Belotti Big-M}}\\
        \cmidrule(lr){4-5}
        & & & Continuous & Binary\\
        \midrule
        \begin{filecontents*}{tables/csv/ccmknap-10-10-big-m-times.csv}
\textsf{mk-10-10},500,3.96s,2.42s,7.45s
,1000,15.81s,5.37s,13.97s
,3000,140.77s,14.35s,41.59s
,5000,381.66s,24.29s,67.13s
\end{filecontents*}
\csvreader[column count=5, no head, late after line=\\, respect none]
        {tables/csv/ccmknap-10-10-big-m-times.csv}
        {1=\acol,2=\bcol,3=\ccol,4=\d,5=\e}
        {\multirow{4}{*}{\acol} & \bcol & \ccol & \d & \e}
        \midrule
        \begin{filecontents*}{tables/csv/ccmknap-20-10-big-m-times.csv}
\textsf{mk-20-10},500,11.94s,4.50s,15.87s
,1000,46.86s,8.87s,52.52s
,3000,374.02s,25.79s,91.37s
,5000,1291.11s,43.58s,237.90s
\end{filecontents*}
\csvreader[column count=5, no head, late after line=\\, respect none]
        {tables/csv/ccmknap-20-10-big-m-times.csv}
        {1=\acol,2=\bcol,3=\ccol,4=\d,5=\e}
        {\multirow{4}{*}{\acol} & \bcol & \ccol & \d & \e}
        \midrule
        \begin{filecontents*}{tables/csv/ccmknap-40-30-big-m-times.csv}
\textsf{mk-40-30},500,221.48s,21.85s,111.83s
,1000,861.83s,44.63s,205.68s
,3000,-,134.08s,626.63s
,5000,-,223.13s,1099.25s
\end{filecontents*}
\csvreader[column count=5, no head, late after line=\\, respect none]
        {tables/csv/ccmknap-40-30-big-m-times.csv}
        {1=\acol,2=\bcol,3=\ccol,4=\d,5=\e}
        {\multirow{4}{*}{\acol} & \bcol & \ccol & \d & \e}
        \bottomrule
    \end{tabular}
    }
\end{table}

As expected, the general method of~\cite{Belotti2016handling} is very fast. In contrast, the method of~\cite{song2014chance}, which is tailored for multi-dimensional knapsack problems, does not scale well with the problem size. It cannot terminate within the time limit for large problems. More precisely, big-M coefficients cannot be computed for the instance~\textsf{mk-40-30} when there are more than~$\card{\scenarios} = 3000$ scenarios. This can be anticipated from the time needed to compute big-M coefficients when~$\card{\scenarios} = 1000$ since increasing the scenario number threefold increases the computation times by a factor of nine. Still, we want to emphasize that our implementation of \cite{song2014chance} is particularly efficient: the times needed to compute the big-M coefficients are approximately three times smaller than those presented in \cite{ahmed2017nonanticipative}.

\subsection{Optimal Solutions and Bounds}
The main experimental results are presented in Table~\ref{table:continuous_results} and Table~\ref{table:binary_results} for instances with continuous and binary variables, respectively. The tables are produced using the method described in~\cite{ahmed2017nonanticipative}. Each row presents the performance metric of all methods averaged over the five perturbed instances. If all instances are solved, we show in the column called~$\avtime$, the average time needed to solve the considered instance to optimality. If at least one instance cannot be solved, we show the average optimality gap over the non-solved instances and the number of solved instances in parentheses. In each row of the table, the best-performing method is shown in bold. \rev{For the APM method~$\Part_{\final}$, we also show the average number of iterations $\aviterset$ as well as the average size of the final partition $\avpartsize$, both rounded to the nearest integer.}

\begin{table}[!ht]
    \caption{Computational comparison of a selection of methods for multi-dimensional knapsack problems with continuous variables.}
    \label{table:continuous_results}
    \centerline{
    \begin{tabular}{c*{8}{r}}
        \toprule
        & & & \multicolumn{2}{c}{$\MILP_{\textbigM}$} & \multicolumn{3}{c}{\APM}\\
        \cmidrule(lr){4-5} \cmidrule(lr){6-8}
        & & & \texttt{Song Big-M} & \texttt{Belotti Big-M} & \multicolumn{3}{c}{$\Part_{\final}$}\\
        \cmidrule(lr){4-4} \cmidrule(lr){5-5} \cmidrule(lr){6-8}
        \multicolumn{1}{c}{Instance} & \multicolumn{1}{c}{$\threshold$} & \multicolumn{1}{c}{$\card{\scenarios}$} & \multicolumn{1}{c}{$\avtime$} & \multicolumn{1}{c}{$\avtime$} & \multicolumn{1}{c}{$\avtime$} & \multicolumn{1}{c}{$\aviterset$} & \multicolumn{1}{c}{\rev{$\avpartsize$}}\\
        \midrule
        \begin{filecontents*}{tables/csv/ccmknap-10-10-1-result-table.csv}
\textsf{mk-10-10},0.1,500,\textbf{24.75s},631.91s,113.94s,33,66
,,1000,0.01\%(4),0.65\%(0),\textbf{1710.02s},66,130
,,3000,0.86\%(0),2.88\%(0),\textbf{0.59\%(0)},177,320
,,5000,1.19\%(0),4.07\%(0),\textbf{0.94\%(0)},287,513
,0.2,500,1901.41s,0.15\%(4),\textbf{1497.49s},59,132
,,1000,0.52\%(0),1.55\%(0),\textbf{0.33\%(0)},89,232
,,3000,1.34\%(0),4.77\%(0),\textbf{1.01\%(0)},203,616
,,5000,2.08\%(0),6.73\%(0),\textbf{1.67\%(0)},313,1011
\end{filecontents*}
\csvreader[column count=9, no head, late after line=\\, respect none]
        {tables/csv/ccmknap-10-10-1-result-table.csv}
        {1=\acol,2=\bcol,3=\ccol,4=\d,5=\e,6=\f,7=\g,8=\h}
        {\multirow{8}{*}{\acol} & \multirow{4}{*}{\bcol} & \ccol & \d & \e & \f & \g & \rev{\h}}
        \midrule
        \begin{filecontents*}{tables/csv/ccmknap-20-10-1-result-table.csv}
\textsf{mk-20-10},0.1,500,\textbf{0.02\%(4)},0.16\%(4),0.03\%(4),56,79
,,1000,0.59\%(0),1.41\%(0),\textbf{0.34\%(0)},86,124
,,3000,1.50\%(0),4.45\%(0),\textbf{1.09\%(0)},195,311
,,5000,1.92\%(0),5.77\%(0),\textbf{1.39\%(0)},283,502
,0.2,500,\textbf{0.33\%(0)},0.89\%(0),0.40\%(0),53,126
,,1000,1.20\%(0),2.49\%(0),\textbf{0.93\%(0)},83,220
,,3000,2.57\%(0),7.75\%(0),\textbf{1.83\%(0)},197,610
,,5000,3.18\%(0),8.78\%(0),\textbf{1.74\%(0)},318,1005
\end{filecontents*}
\csvreader[column count=9, no head, late after line=\\, respect none]
        {tables/csv/ccmknap-20-10-1-result-table.csv}
        {1=\acol,2=\bcol,3=\ccol,4=\d,5=\e,6=\f,7=\g,8=\h}
        {\multirow{8}{*}{\acol} & \multirow{4}{*}{\bcol} & \ccol & \d & \e & \f & \g & \rev{\h}}
        \midrule
        \begin{filecontents*}{tables/csv/ccmknap-40-30-1-result-table.csv}
\textsf{mk-40-30},0.1,500,1.42\%(0),5.29\%(0),\textbf{1.37\%(0)},76,62
,,1000,\textbf{2.88\%(0)},13.50\%(0),4.71\%(0),66,101
,,3000,-,22.58\%(0),\textbf{6.67\%(0)},24,301
,,5000,-,23.01\%(0),\textbf{8.36\%(0)},13,501
,0.2,500,2.87\%(0),8.90\%(0),\textbf{2.19\%(0)},77,111
,,1000,5.22\%(0),24.60\%(0),\textbf{4.31\%(0)},130,205
,,3000,-,41.08\%(0),\textbf{8.87\%(0)},143,601
,,5000,-,42.53\%(0),\textbf{10.71\%(0)},83,1001
\end{filecontents*}
\csvreader[column count=9, no head, late after line=\\, respect none]
        {tables/csv/ccmknap-40-30-1-result-table.csv}
        {1=\acol,2=\bcol,3=\ccol,4=\d,5=\e,6=\f,7=\g,8=\h}
        {\multirow{8}{*}{\acol} & \multirow{4}{*}{\bcol} & \ccol & \d & \e & \f & \g & \rev{\h}}
        \bottomrule
    \end{tabular}
    }
\end{table}

\begin{table}[!ht]
    \caption{Computational comparison of a selection of methods for multi-dimensional knapsack problems with binary variables.}
    \label{table:binary_results}
    \centerline{
    \begin{tabular}{c*{8}{r}}
        \toprule
        & & & \multicolumn{2}{c}{$\MILP_{\textbigM}$} & \multicolumn{3}{c}{\APM}\\
        \cmidrule(lr){4-5} \cmidrule(lr){6-8}
        & & & \texttt{Song Big-M} & \texttt{Belotti Big-M} & \multicolumn{3}{c}{$\Part_{\final}$}\\
        \cmidrule(lr){4-4} \cmidrule(lr){5-5} \cmidrule(lr){6-8}
        \multicolumn{1}{c}{Instance} & \multicolumn{1}{c}{$\threshold$} & \multicolumn{1}{c}{$\card{\scenarios}$} & \multicolumn{1}{c}{$\avtime$} & \multicolumn{1}{c}{$\avtime$} & \multicolumn{1}{c}{$\avtime$} & \multicolumn{1}{c}{$\aviterset$} & \multicolumn{1}{c}{\rev{$\avpartsize$}}\\
        \midrule
        \begin{filecontents*}{tables/csv/ccmknap-10-10-0-result-table.csv}
\textsf{mk-10-10},0.1,500,\textbf{11.16s},21.57s,13.30s,2,52
,,1000,\textbf{39.13s},56.46s,71.50s,5,112
,,3000,163.17s,280.47s,\textbf{120.38s},3,340
,,5000,800.21s,527.89s,\textbf{367.81s},5,599
,0.2,500,\textbf{6.51s},13.35s,8.13s,1,101
,,1000,20.58s,36.88s,\textbf{16.36s},1,211
,,3000,154.22s,151.69s,\textbf{36.81s},1,601
,,5000,406.18s,461.51s,\textbf{165.75s},2,1028
\end{filecontents*}
\csvreader[column count=9, no head, late after line=\\, respect none]
        {tables/csv/ccmknap-10-10-0-result-table.csv}
        {1=\acol,2=\bcol,3=\ccol,4=\d,5=\e,6=\f,7=\g,8=\h}
        {\multirow{8}{*}{\acol} & \multirow{4}{*}{\bcol} & \ccol & \d & \e & \f & \g & \rev{\h}}
        \midrule
        \begin{filecontents*}{tables/csv/ccmknap-20-10-0-result-table.csv}
\textsf{mk-20-10},0.1,500,\textbf{30.28s},51.48s,134.38s,9,69
,,1000,\textbf{91.65s},262.26s,333.44s,12,135
,,3000,\textbf{635.70s},2538.93s,2726.06s,19,401
,,5000,\textbf{1999.27s},2.73\%(2),0.27\%(2),17,680
,0.2,500,\textbf{34.98s},81.76s,250.84s,14,137
,,1000,\textbf{107.58s},668.80s,940.53s,18,267
,,3000,\textbf{1780.72s},5.18\%(0),0.54\%(1),24,750
,,5000,2.74\%(0),6.79\%(0),\textbf{1.34\%(0)},19,1094
\end{filecontents*}
\csvreader[column count=9, no head, late after line=\\, respect none]
        {tables/csv/ccmknap-20-10-0-result-table.csv}
        {1=\acol,2=\bcol,3=\ccol,4=\d,5=\e,6=\f,7=\g,8=\h}
        {\multirow{8}{*}{\acol} & \multirow{4}{*}{\bcol} & \ccol & \d & \e & \f & \g & \rev{\h}}
        \midrule
        \begin{filecontents*}{tables/csv/ccmknap-40-30-0-result-table.csv}
\textsf{mk-40-30},0.1,500,\textbf{351.52s},2503.72s,0.39\%(4),9,74
,,1000,\textbf{1295.05s},11.40\%(0),1.76\%(0),7,136
,,3000,-,41.22\%(0),\textbf{3.88\%(0)},5,315
,,5000,-,47.21\%(0),\textbf{5.76\%(0)},4,517
,0.2,500,\textbf{400.16s},18.37\%(1),1.60\%(0),9,136
,,1000,\textbf{2054.95s},30.80\%(0),4.03\%(0),10,213
,,3000,-,53.35\%(0),\textbf{6.14\%(0)},10,603
,,5000,-,59.60\%(0),\textbf{5.72\%(0)},15,1008
\end{filecontents*}
\csvreader[column count=9, no head, late after line=\\, respect none]
        {tables/csv/ccmknap-40-30-0-result-table.csv}
        {1=\acol,2=\bcol,3=\ccol,4=\d,5=\e,6=\f,7=\g,8=\h}
        {\multirow{8}{*}{\acol} & \multirow{4}{*}{\bcol} & \ccol & \d & \e & \f & \g & \rev{\h}}
        \bottomrule
    \end{tabular}
    }
\end{table}

As can be observed in Table~\ref{table:continuous_results}, the method~$\Part_{\final}$ performs the best among the pool of compared methods when continuous variables are considered. More specifically,~$\Part_{\final}$ is the best-performing method for 20 out of the 24 considered instances. For the remaining 4 instances the method \texttt{Song Big-M} performs best. Thanks to the extremely tight big-M coefficients that are produced with this method, the time needed to solve Model~\eqref{eq:cclp-reform-bigm} and the final optimality gaps are always smaller compared to the \texttt{Belotti Big-M} method.

The results in Table~\ref{table:binary_results} suggest that there is no dominating algorithm when binary variables are considered. In particular, the method \texttt{Song Big-M} makes it possible to solve most of the instances that have a large number of scenarios to optimality when binary variables are considered. Nevertheless, the method \texttt{Song Big-M} does not produce valid results when the time limit is reached for instance~\textsf{mk-40-30} with~$\card{\scenarios} \geq 3000$. As explained earlier, due to the large amount of constraints and scenarios, too much time is consumed for computing the big-M coefficients. For these same instances the method~$\Part_{\final}$ significantly reduces the optimality gap compared to the \texttt{Belotti Big-M} benchmark. We also observe that, for instance~\textsf{mk-10-10}, only a few iterations are needed to solve the problem to optimality with method~$\Part_{\final}$. This demonstrates the strength of the partition presented in Section~\ref{sec:init}. In fact, the optimal solution of the original CCSP is, in some cases, found directly using the first partition.

Finally, we draw attention to the fact that when the number of scenarios is large,~$\Part_{\final}$ consistently performs better than other methods for both continuous and binary variables. This is the strength of the \APM. By design, it reduces the amount of binary variables that are considered and, as a consequence, scales better with the size of the scenario set.

\subsection{Partitioning Strategies}
\rev{We now study the contribution of the partitioning strategies introduced in Section~\ref{sec:strong-part}. We introduce three parameterizations of the \APM that follow the structure of Algorithm~\ref{alg:adaptive-partitioning} but have different strategies for initializing and refining the partition. We also investigate the benefit of merging.}

\rev{The first method, denoted by $\Part_{\random}$, is a naive parameterization of the APM. It is based on restricting Algorithm~\ref{alg:adaptive-partitioning} to refinement operations. Each subset of the partition to be refined is chosen at random. The dispatch of scenarios to child subsets is also chosen at random. Nevertheless, the outline of Algorithm~\ref{alg:refinement}, \ie, how many scenarios are selected for refinement and which scenarios are considered for refinement, is respected. The initial partition of $\Part_{\random}$ is obtained by randomly dispatching the scenarios while keeping their size balanced. The second method, denoted by $\Part_{\init}$, is based on the result of Section~\ref{sec:init} for obtaining an initial partition but otherwise uses the same implementation as $\Part_{\random}$. The third method, denoted by $\Part_{\violmerge}$, splits the subset that contains the most violated scenario and allocates scenarios to $\newrefinedleft$ and $\newrefinedright$ in decreasing order of their violation. That is, the most violated scenario is allocated to $\newrefinedleft$, the second-most violated to the $\newrefinedright$, and so on. Further, this last method is allowed to merge subsets.}

\rev{These methods are run on the same instances as in the previous experiments, but restricted to $\tau = 0.2$ since they are the most challenging instances according to the previous experiments. The experiment results are shown in Table~\ref{table:continuous_ablation} and Table~\ref{table:binary_ablation} for continuous and binary variables, respectively.

Although the method $\Part_{\random}$ is guaranteed to terminate finitely, this method is less competitive compared to the remaining parameterizations. Still, it is interesting to note that this naive method yields lower optimality gaps than the \texttt{Belotti Big-M} benchmark when the number of scenarios is large.

The results show that the initial partition described in Section~\ref{sec:init} provides significant improvement in terms of solution time and optimality gap as can be observed when comparing $\Part_\init$ and $\Part_\random$. However, we do not observe a significant difference between the results of $\Part_\init$ and $\Part_\violmerge$. This is surprising since $\Part_\violmerge$ has a more elaborate refinement strategy and is allowed to merge. Yet, we observe that $\Part_\violmerge$ does, in fact, perform very few merge operations as the number of iterations and the partition size are very close to the ones of $\Part_\init$.

However, our final algorithm $\Part_\final$ performs significantly more iterations and tends to have smaller final partitions than all other partitioning algorithms. This suggests that the refinement strategies that promote merging, introduced in Section~\ref{sec:accurate-split}, are beneficial. A key insight of this experiment is that merging alone is not sufficient to keep the resulting partitions small: the refinement strategy needs to be aligned with the merging strategy to keep the considered partitions small.}

\begin{table}[!ht]
    \caption{Sensitivity analysis of APM for multi-dimensional knapsack problems with continuous variables.}
    \label{table:continuous_ablation}
    \centerline{
    \footnotesize
    \begin{tabular}{c*{14}{r}}
        \toprule
        & & \multicolumn{3}{c}{$\Part_{\random}$} & \multicolumn{3}{c}{$\Part_{\init}$} & \multicolumn{3}{c}{$\Part_{\violmerge}$} & \multicolumn{3}{c}{$\Part_{\final}$}\\
        \cmidrule(lr){3-5} \cmidrule(lr){6-8} \cmidrule(lr){9-11} \cmidrule(lr){12-14}
        \multicolumn{1}{c}{Instance} & \multicolumn{1}{c}{$\card{\scenarios}$} & \multicolumn{1}{c}{$\avtime$} & \multicolumn{1}{c}{$\aviterset$} & \multicolumn{1}{c}{$\avpartsize$} & \multicolumn{1}{c}{$\avtime$} & \multicolumn{1}{c}{$\aviterset$} & \multicolumn{1}{c}{$\avpartsize$} & \multicolumn{1}{c}{$\avtime$} & \multicolumn{1}{c}{$\aviterset$} & \multicolumn{1}{c}{$\avpartsize$} & \multicolumn{1}{c}{$\avtime$} & \multicolumn{1}{c}{$\aviterset$} & \multicolumn{1}{c}{$\avpartsize$}\\
        \midrule
        \begin{filecontents*}{tables/csv/sens-ccmknap-10-10-1-result-table.csv}
\textsf{mk-10-10},500,0.03\%(4),69,169,0.02\%(4),44,144,2003.62s,43,142,\textbf{1497.49s},59,132
,1000,1.97\%(0),43,243,0.84\%(0),35,235,0.69\%(0),41,239,\textbf{0.33\%(0)},89,232
,3000,5.24\%(0),23,623,2.14\%(0),17,617,2.15\%(0),18,617,\textbf{1.01\%(0)},203,616
,5000,7.70\%(0),39,1039,2.77\%(0),13,1013,2.71\%(0),15,1013,\textbf{1.67\%(0)},313,1011
\end{filecontents*}
\csvreader[column count=17, no head, late after line=\\, respect none]
        {tables/csv/sens-ccmknap-10-10-1-result-table.csv}{1=\acol,2=\bcol,3=\ccol,4=\dcol,5=\ecol,6=\fcol,7=\gcol,8=\hcol,9=\icol,10=\jcol,11=\kcol,12=\lcol,13=\mcol,14=\ncol}
        {\multirow{4}{*}{\acol} & \bcol & \ccol & \dcol & \ecol & \fcol & \gcol & \hcol & \icol & \jcol & \kcol & \lcol & \mcol & \ncol}
        \midrule
        \begin{filecontents*}{tables/csv/sens-ccmknap-20-10-1-result-table.csv}
\textsf{mk-20-10},500,1.31\%(0),38,138,0.94\%(0),30,130,0.55\%(0),35,132,\textbf{0.40\%(0)},53,126
,1000,3.79\%(0),23,223,1.84\%(0),19,219,1.85\%(0),20,219,\textbf{0.93\%(0)},83,220
,3000,5.90\%(0),16,616,3.10\%(0),12,612,3.01\%(0),15,613,\textbf{1.83\%(0)},197,610
,5000,6.77\%(0),15,1015,3.00\%(0),9,1009,2.88\%(0),12,1009,\textbf{1.74\%(0)},318,1005
\end{filecontents*}
\csvreader[column count=17, no head, late after line=\\, respect none]
        {tables/csv/sens-ccmknap-20-10-1-result-table.csv}{1=\acol,2=\bcol,3=\ccol,4=\dcol,5=\ecol,6=\fcol,7=\gcol,8=\hcol,9=\icol,10=\jcol,11=\kcol,12=\lcol,13=\mcol,14=\ncol}
        {\multirow{4}{*}{\acol} & \bcol & \ccol & \dcol & \ecol & \fcol & \gcol & \hcol & \icol & \jcol & \kcol & \lcol & \mcol & \ncol}
        \midrule
        \begin{filecontents*}{tables/csv/sens-ccmknap-40-30-1-result-table.csv}
\textsf{mk-40-30},500,4.98\%(0),14,114,4.48\%(0),13,113,3.83\%(0),14,113,\textbf{2.19\%(0)},77,111
,1000,10.74\%(0),7,207,7.92\%(0),7,207,8.62\%(0),7,207,\textbf{4.31\%(0)},130,205
,3000,20.22\%(0),5,605,14.46\%(0),3,603,14.61\%(0),4,603,\textbf{8.87\%(0)},143,601
,5000,27.93\%(0),3,1003,16.57\%(0),3,1003,14.84\%(0),4,1003,\textbf{10.71\%(0)},83,1001
\end{filecontents*}
\csvreader[column count=17, no head, late after line=\\, respect none]
        {tables/csv/sens-ccmknap-40-30-1-result-table.csv}{1=\acol,2=\bcol,3=\ccol,4=\dcol,5=\ecol,6=\fcol,7=\gcol,8=\hcol,9=\icol,10=\jcol,11=\kcol,12=\lcol,13=\mcol,14=\ncol}
        {\multirow{4}{*}{\acol} & \bcol & \ccol & \dcol & \ecol & \fcol & \gcol & \hcol & \icol & \jcol & \kcol & \lcol & \mcol & \ncol}
        \bottomrule
    \end{tabular}
    }
\end{table}

\begin{table}[!ht]
    \caption{Sensitivity analysis of APM for multi-dimensional knapsack problems with binary variables.}
    \label{table:binary_ablation}
    \centerline{
    \footnotesize
    \begin{tabular}{c*{14}{r}}
        \toprule
        & & \multicolumn{3}{c}{$\Part_{\random}$} & \multicolumn{3}{c}{$\Part_{\init}$} & \multicolumn{3}{c}{$\Part_{\violmerge}$} & \multicolumn{3}{c}{$\Part_{\final}$}\\
        \cmidrule(lr){3-5} \cmidrule(lr){6-8} \cmidrule(lr){9-11} \cmidrule(lr){12-14}
        \multicolumn{1}{c}{Instance} & \multicolumn{1}{c}{$\card{\scenarios}$} & \multicolumn{1}{c}{$\avtime$} & \multicolumn{1}{c}{$\aviterset$} & \multicolumn{1}{c}{$\avpartsize$} & \multicolumn{1}{c}{$\avtime$} & \multicolumn{1}{c}{$\aviterset$} & \multicolumn{1}{c}{$\avpartsize$} & \multicolumn{1}{c}{$\avtime$} & \multicolumn{1}{c}{$\aviterset$} & \multicolumn{1}{c}{$\avpartsize$} & \multicolumn{1}{c}{$\avtime$} & \multicolumn{1}{c}{$\aviterset$} & \multicolumn{1}{c}{$\avpartsize$}\\
        \midrule
        \begin{filecontents*}{tables/csv/sens-ccmknap-10-10-0-result-table.csv}
\textsf{mk-10-10},500,35.51s,6,110,8.35s,1,101,\textbf{6.60s},1,101,8.13s,1,101
,1000,67.56s,6,220,\textbf{15.15s},1,211,15.55s,1,211,16.36s,1,211
,3000,0.83\%(4),8,626,36.06s,1,601,\textbf{34.80s},1,601,36.81s,1,601
,5000,931.61s,14,1170,189.12s,3,1029,\textbf{143.19s},2,1028,165.75s,2,1028
\end{filecontents*}
\csvreader[column count=17, no head, late after line=\\, respect none]
        {tables/csv/sens-ccmknap-10-10-0-result-table.csv}{1=\acol,2=\bcol,3=\ccol,4=\dcol,5=\ecol,6=\fcol,7=\gcol,8=\hcol,9=\icol,10=\jcol,11=\kcol,12=\lcol,13=\mcol,14=\ncol}
        {\multirow{4}{*}{\acol} & \bcol & \ccol & \dcol & \ecol & \fcol & \gcol & \hcol & \icol & \jcol & \kcol & \lcol & \mcol & \ncol}
        \midrule
        \begin{filecontents*}{tables/csv/sens-ccmknap-20-10-0-result-table.csv}
\textsf{mk-20-10},500,466.42s,18,159,325.45s,12,149,264.81s,11,144,\textbf{250.84s},14,137
,1000,0.03\%(4),22,312,1284.55s,14,287,1256.94s,14,280,\textbf{940.53s},18,267
,3000,2.61\%(0),21,683,1.01\%(0),13,751,0.97\%(0),13,732,\textbf{0.54\%(1)},24,750
,5000,4.21\%(0),15,1057,1.81\%(0),7,1069,1.84\%(0),9,1091,\textbf{1.34\%(0)},19,1094
\end{filecontents*}
\csvreader[column count=17, no head, late after line=\\, respect none]
        {tables/csv/sens-ccmknap-20-10-0-result-table.csv}{1=\acol,2=\bcol,3=\ccol,4=\dcol,5=\ecol,6=\fcol,7=\gcol,8=\hcol,9=\icol,10=\jcol,11=\kcol,12=\lcol,13=\mcol,14=\ncol}
        {\multirow{4}{*}{\acol} & \bcol & \ccol & \dcol & \ecol & \fcol & \gcol & \hcol & \icol & \jcol & \kcol & \lcol & \mcol & \ncol}
        \midrule
        \begin{filecontents*}{tables/csv/sens-ccmknap-40-30-0-result-table.csv}
\textsf{mk-40-30},500,2.57\%(0),7,127,2.19\%(0),7,129,1.97\%(0),7,130,\textbf{1.60\%(0)},9,136
,1000,5.98\%(0),5,208,4.63\%(0),6,228,4.98\%(0),6,220,\textbf{4.03\%(0)},10,213
,3000,21.11\%(0),2,602,10.77\%(0),2,603,9.76\%(0),2,603,\textbf{6.14\%(0)},10,603
,5000,23.74\%(0),2,1002,11.47\%(0),2,1002,10.43\%(0),4,1002,\textbf{5.72\%(0)},15,1008
\end{filecontents*}
\csvreader[column count=17, no head, late after line=\\, respect none]
        {tables/csv/sens-ccmknap-40-30-0-result-table.csv}{1=\acol,2=\bcol,3=\ccol,4=\dcol,5=\ecol,6=\fcol,7=\gcol,8=\hcol,9=\icol,10=\jcol,11=\kcol,12=\lcol,13=\mcol,14=\ncol}
        {\multirow{4}{*}{\acol} & \bcol & \ccol & \dcol & \ecol & \fcol & \gcol & \hcol & \icol & \jcol & \kcol & \lcol & \mcol & \ncol}
        \bottomrule
    \end{tabular}
    }
\end{table}

\begin{remark}
  \rev{We ran experiments with a variation of $\Part_\violmerge$ without merge operations. We observed no significant difference in terms of solution time, optimality gap, number of iterations, and partition size compared to $\Part_\violmerge$ and $\Part_\init$. We do not include these results to keep this section concise.}
\end{remark}

\revtwo{\begin{remark}
  We provide a detailed analysis of what happens when the number of iterations is increased beyond the minimum value~$\extrasize$ in Algorithm~\ref{alg:refinement}. In that case, we observe a faster termination of the \APM with a larger final partition. This study is available in Appendix~\ref{sec:trade-betw-part}.
\end{remark}}

\subsection{Detailed Analysis}
To study the behavior of~$\Part_{\final}$ and identify the strengths and weaknesses of this method, we provide a detailed analysis on a single instance:~\textsf{mk-20-10} with~$\abs{\scenarios} = 1000$ scenarios.

\subsubsection{Partition Size}
The number of subsets that compose the partition as a function of the number of iterations is shown in Figure~\ref{fig:plot_part_size}. When applying method~$\Part_{\final}$, \rev{the partitions first stay of minimal size for many iterations and then increase steadily with the iterations. In fact, Figure~\ref{fig:plot_part_size} suggests that $\Part_{\final}$ follows two phases: a first phase in which many merging operations are performed, and a second phase in which few or no merging operations are performed. These phases are not hard-coded but emerge naturally from the design of~$\Part_{\final}$.}

\rev{In the first phase,} merge operations are performed at each iteration, which forces the partition to stay as small as possible. In this phase, the reduced-size model is solved very efficiently. Indeed, when the size of the partition is minimal, \ie, it is equal to~$\lfloor \threshold \abs{\scenarios} \rfloor + 1$, the chance constraint~\eqref{eq:chance-cons-part} reduces to~$\sum_{\partition\in\partitions} \indvar_\partition \geq 1$. Hence, the optimal solution of the reduced model is given by the point that corresponds to~$\min_{\scenario \in \scenarios} \singlecost_\scenario$.

\rev{The second phase begins when the size of the partition starts growing. During these remaining iterations,} fewer merging operations are carried out. As a result, the size of the partition increases steadily, and solving the reduced-size model requires more time in each iteration. Overall, more iterations are carried out when continuous variables are considered compared to when binary variables are considered in the first phase of~$\Part_{\final}$, as could already be observed in Tables~\ref{table:continuous_results}~and~\ref{table:binary_results}. Moreover, we observe that each time a refinement operation is carried out for an instance with continuous variables, exactly one additional subset is created. In contrast, in the binary case, the minimal-size refinement tends to create much more than one additional subset. This can be explained by the fact that, due to the integrality restrictions on the decision variables, the optimal solutions satisfy more sets~$\probset^\partition$. The fact that partitions that necessitate \SI{75}{\percent} less binary variables can yield optimal solutions to the original CCSP demonstrates the value of adaptive methods for CCSPs.

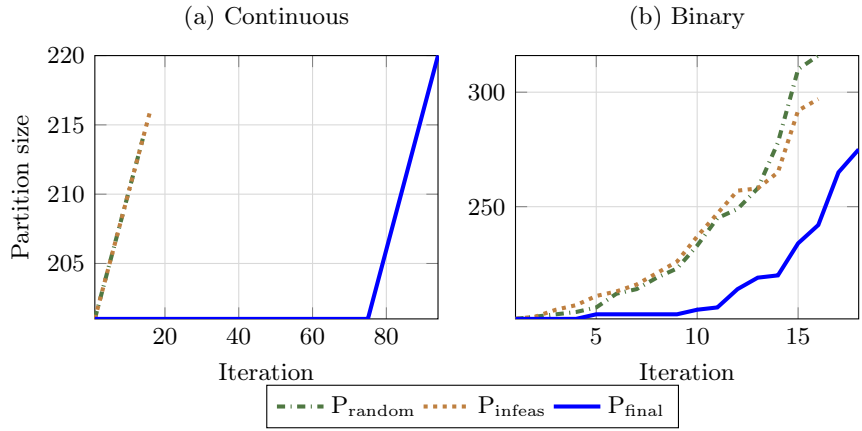
\begin{figure}[ht!]
    \centering
    \resizebox{0.9\linewidth}{!}{\begin{tikzpicture}
\begin{groupplot}[
    group style={
        group name=my plots,
        group size=2 by 1,
        xlabels at=edge bottom,
        ylabels at=edge left},
    height = 5cm,
    width  = 6cm,
    enlarge x limits = 0,
    enlarge y limits = 0,
    xlabel = {Iteration},
    ylabel = {Partition size},% $\car{\partitions}$},
    xmin = 1,
    ]

    \nextgroupplot[title = {(a) Continuous},
                   font = \small]
    %%%%%%%%%%% P_random %%%%%%%%%%%
    \addplot+[randomStyle] table [x index = {0}, y index = {5}, col sep=comma]{plots/csv/ccmknap-20-10-1000-2-20-1-3-iter.csv};
    %%%%%%%%%%% P_infeas %%%%%%%%%%%
    \addplot+[infeastyle] table [x index = {0}, y index = {5}, col sep=comma]{plots/csv/ccmknap-20-10-1000-2-20-1-7-iter.csv};
    %%%%%%%%%%% P_final %%%%%%%%%%%
    \addplot+[finalStyle] table [x index = {0}, y index = {5}, col sep=comma]{plots/csv/ccmknap-20-10-1000-2-20-1-4-iter.csv};

    \nextgroupplot[title = {(b) Binary},
                   font = \small,
                   legend style={at={(-0.125,-0.25)}, anchor=north},
                   legend columns=3]
    %%%%%%%%%%% P_random %%%%%%%%%%%
    \addplot+[randomStyle] table [x index = {0}, y index = {5}, col sep=comma]{plots/csv/ccmknap-20-10-1000-2-20-0-3-iter.csv};
    \addlegendentry{$\Part_{\random}$}
    %%%%%%%%%%% P_infeas %%%%%%%%%%%
    \addplot+[infeastyle] table [x index = {0}, y index = {5}, col sep=comma]{plots/csv/ccmknap-20-10-1000-2-20-0-7-iter.csv};
    \addlegendentry{$\Part_{\violmerge}$}
    %%%%%%%%%%% P_final %%%%%%%%%%%
    \addplot+[finalStyle] table [x index = {0}, y index = {5}, col sep=comma]{plots/csv/ccmknap-20-10-1000-2-20-0-4-iter.csv};
    \addlegendentry{$\Part_{\final}$}

\end{groupplot}
\end{tikzpicture}}
    \caption{Partition size over iterations for \textsf{mk-20-10} with $1000$ scenarios with continuous and binary variables. \rev{The methods $P_\random$ and $P_\violmerge$ are indistinguishable in the (a)~continuous case.}}
    \label{fig:plot_part_size}
\end{figure}

\subsubsection{Bound Evolution with Refinement and Merging}

In Figures~\ref{fig:bound_evol_cont} and~\ref{fig:bound_evol_bin}, we examine how the lower and upper bounds, along with the iteration count, evolve over time. Figure~\ref{fig:bound_evol_cont}~(a) shows how the method~$\Part_{\final}$ achieves a computational advantage over the other methods when continuous variables are considered: it quickly finds a tight lower bound. This is explained by the large number of iterations performed in the first phase of the final \APM when the partition size stays minimal. Figure~\ref{fig:bound_evol_cont}~(a) shows that the first phase lasts around \SI{450}{\second}. Once the second phase is reached, the partition size increases, and more time is required to solve each optimization problem. This can be observed in Figure~\ref{fig:bound_evol_cont}~(b) as the number of iterations performed over time decreases after \SI{450}{\second}. Figure~\ref{fig:bound_evol_bin}~(a) also highlights the computational advantage of the method \texttt{Song Big-M} when binary variables are considered. Thanks to its very tight big-M coefficients, this method quickly produces a feasible solution and can close the optimality gap in a short amount of time. On the other hand,~$\Part_{\final}$ stays in the first phase for several iterations, resulting in a longer computation time.

\begin{figure}[ht!]
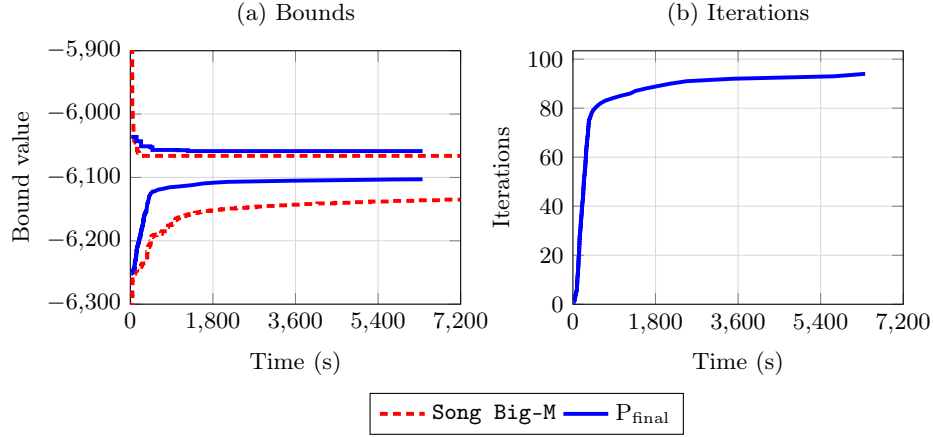

    \centering
    \resizebox{\linewidth}{!}{
        \plotresultsupperlower{plots/csv/ccmknap-20-10-1000-2-20-1-1-iter.csv}{plots/csv/ccmknap-20-10-1000-2-20-1-4-iter.csv}{7200}}
    \caption{Convergence plot and number of iterations performed over time on \textsf{mk-20-10} with $1000$ scenarios and continuous variables.}
    \label{fig:bound_evol_cont}
\end{figure}

\begin{figure}[ht!]
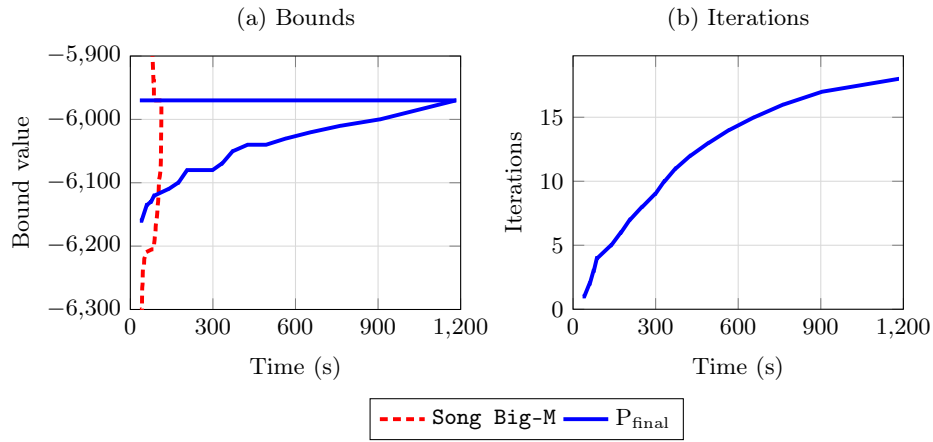

    \centering
    \resizebox{\linewidth}{!}{
        \plotresultsupperlower{plots/csv/ccmknap-20-10-1000-2-20-0-1-iter.csv}{plots/csv/ccmknap-20-10-1000-2-20-0-4-iter.csv}{1200}}
    \caption{Convergence plot and number of iterations performed over time on \textsf{mk-20-10} with $1000$ scenarios and binary variables. \rev{The method \texttt{Song Big-M} closes the gap in 112 seconds}}
    \label{fig:bound_evol_bin}
\end{figure}

%%% Local Variables:
%%% mode: latex
%%% TeX-master: "../main"
%%% End:

\section{Conclusion}
\label{sec:conclusion}

This paper introduces a method for solving chance-constrained problems with finite support based on iteratively partitioning the scenario set and solving reduced models. The method provides valid \rev{lower} bounds on the original stochastic problem and is guaranteed to recover its optimal solution in a finite number of iterations. The key idea of the proposed method is to modify the partition by excluding previously found solutions. We use mathematical arguments to find modifications of the partition that ensure a strict increase of the lower bounds while keeping the size of the reduced model as small as possible.

The proposed method surpasses state-of-the-art benchmarks on standard instances with continuous variables and performs comparably on instances with binary variables. Its scalability improves notably when solving instances with a large number of scenarios and constraints.

Future research avenues include extensions for non-equiprobable scenarios, adapting the partitioning method for constraint aggregations within subsets, or leveraging partitions to enhance techniques that exploit the quantile bound. \rev{More generally, new methods to modify partitions might be investigated that do not rely on refining and merging. Another promising direction is whether bound improvements through refinement and merging can be performed within a single branch-and-bound tree. This removes the need to solve a new problem from scratch at each iteration and makes it possible to reuse computational information from previous iterations. Finally, a promising direction to improve the understanding of CCSPs is studying the existence of small completely sufficient partitions.}

%%% Local Variables:
%%% mode: latex
%%% TeX-master: "../main"
%%% End:

\section*{Acknowledgments}
\label{sec:acknowledgements}
The authors want to thank Youssouf Emine for his help with the highly optimized implementation of a benchmark method, as well as Martine Labb\'{e} for her helpful comments on an earlier version of this manuscript. The support of IVADO, the Canada First Research Excellence Fund (Apog\'ee/CFREF), as well as the computational infrastructure provided by Compute Canada are also gratefully acknowledged.

%%% Local Variables:
%%% mode: latex
%%% TeX-master: "main"
%%% End:

\printbibliography

\appendix

\section{Beyond the Minimum Number of Splits}
\label{sec:trade-betw-part}
\revtwo{In this appendix, we experiment with a new variant of \APM denoted by $\Part_\beta$. This variant performs more than the minimum number of refinement operations~$\extrasize$ required at each iteration. The parameter $\beta$ represents the relative increase in the number of splits compared to the minimum value. In other words,~$\Part_\beta$ creates~$\mu_\beta = \mu \cdot (1 + \beta / 100)$ new subsets in each iteration. We consider~$\beta \in \{0.5, 1.0, 2.0\}$.

Recall that our refinement procedure splits subsets that are in the set~$\refined_2 (x) = \{\newrefined \in \refined: \card{\newrefined \cap \infscenarios(\lbsol)} \geq 2\}$ and contain at least two infeasible scenarios. In practice,~$\card{\refined_2 (x)}$ may be smaller than~$\mu_\beta$. Hence, we perform exactly $\min (\extrasize_\beta, \card{\refined_2 (x)})$ splits. Apart from the number of splits, the algorithm is kept the same as $\Part_\final$.}
\begin{table}[!ht]
    \caption{\revtwo{Results of $\Part_\beta$ for multi-dimensional knapsack problems} \revtwo{with continuous variables.}}
    \label{table:continuous_splits}
    \centerline{
    \scriptsize
    \revtwo{\begin{tabular}{c*{20}{r}}
        \toprule
        & & \multicolumn{3}{c}{$\Part_\final$} & \multicolumn{3}{c}{$\beta = 50\%$} & \multicolumn{3}{c}{$\beta = 100\%$} & \multicolumn{3}{c}{$\beta = 200\%$}\\
        \cmidrule(lr){3-5} \cmidrule(lr){6-8} \cmidrule(lr){9-11} \cmidrule(lr){12-14}
        \multicolumn{1}{c}{Instance} & \multicolumn{1}{c}{$\card{\scenarios}$} & \multicolumn{1}{c}{$\avtime$} & \multicolumn{1}{c}{$\aviterset$} & \multicolumn{1}{c}{$\avpartsize$} & \multicolumn{1}{c}{$\avtime$} & \multicolumn{1}{c}{$\aviterset$} & \multicolumn{1}{c}{$\avpartsize$} & \multicolumn{1}{c}{$\avtime$} & \multicolumn{1}{c}{$\aviterset$} & \multicolumn{1}{c}{$\avpartsize$} & \multicolumn{1}{c}{$\avtime$} & \multicolumn{1}{c}{$\aviterset$} & \multicolumn{1}{c}{$\avpartsize$}\\
        \midrule
        \begin{filecontents*}{tables/csv/split-ccmknap-10-10-1-result-table.csv}
\textsf{mk-10-10},500,1497.49s,59,132,1164.14s,27,136,1177.12s,27,136,\textbf{825.71s},17,135
,1000,0.33\%(0),89,232,0.40\%(0),34,236,0.40\%(0),34,236,\textbf{0.37\%(1)},23,240
,3000,\textbf{1.01\%(0)},203,616,1.31\%(0),67,617,1.31\%(0),67,617,1.35\%(0),41,618
,5000,\textbf{1.67\%(0)},313,1011,2.40\%(0),49,1008,2.40\%(0),49,1008,2.34\%(0),33,1009
\end{filecontents*}
\csvreader[column count=20, no head, late after line=\\, respect none]
        {tables/csv/split-ccmknap-10-10-1-result-table.csv}{1=\acol,2=\bcol,3=\ccol,4=\dcol,5=\ecol,6=\fcol,7=\gcol,8=\hcol,9=\icol,10=\jcol,11=\kcol,12=\lcol,13=\mcol,14=\ncol,15=\ocol,16=\pcol,17=\qcol,18=\rcol,19=\scol,20=\tcol}
        {\multirow{4}{*}{\acol} & \bcol & \ccol & \dcol & \ecol & \fcol & \gcol & \hcol & \icol & \jcol & \kcol & \lcol & \mcol & \ncol}
        \midrule
        \begin{filecontents*}{tables/csv/split-ccmknap-20-10-1-result-table.csv}
\textsf{mk-20-10},500,0.40\%(0),53,126,0.44\%(0),23,131,0.44\%(0),23,131,\textbf{0.37\%(1)},15,133
,1000,\textbf{0.93\%(0)},83,220,1.00\%(0),26,223,1.00\%(0),26,223,0.97\%(0),18,227
,3000,\textbf{1.83\%(0)},197,610,2.38\%(0),44,608,2.38\%(0),44,608,2.33\%(0),28,609
,5000,\textbf{1.74\%(0)},318,1005,2.69\%(0),28,1003,2.69\%(0),28,1003,2.61\%(0),18,1004
\end{filecontents*}
\csvreader[column count=20, no head, late after line=\\, respect none]
        {tables/csv/split-ccmknap-20-10-1-result-table.csv}{1=\acol,2=\bcol,3=\ccol,4=\dcol,5=\ecol,6=\fcol,7=\gcol,8=\hcol,9=\icol,10=\jcol,11=\kcol,12=\lcol,13=\mcol,14=\ncol,15=\ocol,16=\pcol,17=\qcol,18=\rcol,19=\scol,20=\tcol}
        {\multirow{4}{*}{\acol} & \bcol & \ccol & \dcol & \ecol & \fcol & \gcol & \hcol & \icol & \jcol & \kcol & \lcol & \mcol & \ncol}
        \midrule
        \begin{filecontents*}{tables/csv/split-ccmknap-40-30-1-result-table.csv}
\textsf{mk-40-30},500,\textbf{2.19\%(0)},77,111,2.42\%(0),14,113,2.42\%(0),14,113,2.66\%(0),9,114
,1000,\textbf{4.31\%(0)},130,205,5.46\%(0),13,207,5.46\%(0),13,207,5.36\%(0),8,208
,3000,\textbf{8.87\%(0)},143,601,14.52\%(0),3,603,14.52\%(0),3,603,18.38\%(0),2,604
,5000,\textbf{10.71\%(0)},83,1001,17.70\%(0),2,1003,17.70\%(0),2,1003,17.70\%(0),2,1004
\end{filecontents*}
\csvreader[column count=20, no head, late after line=\\, respect none]
        {tables/csv/split-ccmknap-40-30-1-result-table.csv}{1=\acol,2=\bcol,3=\ccol,4=\dcol,5=\ecol,6=\fcol,7=\gcol,8=\hcol,9=\icol,10=\jcol,11=\kcol,12=\lcol,13=\mcol,14=\ncol,15=\ocol,16=\pcol,17=\qcol,18=\rcol,19=\scol,20=\tcol}
        {\multirow{4}{*}{\acol} & \bcol & \ccol & \dcol & \ecol & \fcol & \gcol & \hcol & \icol & \jcol & \kcol & \lcol & \mcol & \ncol}
        \bottomrule
    \end{tabular}}
    }
  \end{table}

%%% Local Variables:
%%% mode: latex
%%% TeX-master: "../main"
%%% End:

\begin{table}[!ht]
    \caption{\revtwo{Results of $\Part_\beta$ for multi-dimensional knapsack problems} \revtwo{with binary variables.}}
    \label{table:binary_splits}
    \centerline{
      \scriptsize
      \revtwo{\begin{tabular}{c*{20}{r}}
        \toprule
        & & \multicolumn{3}{c}{$\Part_\final$} & \multicolumn{3}{c}{$\beta = 50\%$} & \multicolumn{3}{c}{$\beta = 100\%$} & \multicolumn{3}{c}{$\beta = 200\%$}\\
        \cmidrule(lr){3-5} \cmidrule(lr){6-8} \cmidrule(lr){9-11} \cmidrule(lr){12-14}
        \multicolumn{1}{c}{Instance} & \multicolumn{1}{c}{$\card{\scenarios}$} & \multicolumn{1}{c}{$\avtime$} & \multicolumn{1}{c}{$\aviterset$} & \multicolumn{1}{c}{$\avpartsize$} & \multicolumn{1}{c}{$\avtime$} & \multicolumn{1}{c}{$\aviterset$} & \multicolumn{1}{c}{$\avpartsize$} & \multicolumn{1}{c}{$\avtime$} & \multicolumn{1}{c}{$\aviterset$} & \multicolumn{1}{c}{$\avpartsize$} & \multicolumn{1}{c}{$\avtime$} & \multicolumn{1}{c}{$\aviterset$} & \multicolumn{1}{c}{$\avpartsize$}\\
        \midrule
        \begin{filecontents*}{tables/csv/split-ccmknap-10-10-0-result-table.csv}
\textsf{mk-10-10},500,8.13s,1,101,6.32s,1,101,\textbf{6.08s},1,101,6.10s,1,101
,1000,\textbf{16.36s},1,211,19.53s,1,211,19.25s,1,211,19.07s,1,211
,3000,36.81s,1,601,36.23s,1,601,37.09s,1,601,\textbf{36.00s},1,601
,5000,165.75s,2,1028,182.86s,2,1030,\textbf{162.78s},2,1031,165.20s,2,1042
\end{filecontents*}
\csvreader[column count=20, no head, late after line=\\, respect none]
        {tables/csv/split-ccmknap-10-10-0-result-table.csv}{1=\acol,2=\bcol,3=\ccol,4=\dcol,5=\ecol,6=\fcol,7=\gcol,8=\hcol,9=\icol,10=\jcol,11=\kcol,12=\lcol,13=\mcol,14=\ncol,15=\ocol,16=\pcol,17=\qcol,18=\rcol,19=\scol,20=\tcol}
        {\multirow{4}{*}{\acol} & \bcol & \ccol & \dcol & \ecol & \fcol & \gcol & \hcol & \icol & \jcol & \kcol & \lcol & \mcol & \ncol}
        \midrule
        \begin{filecontents*}{tables/csv/split-ccmknap-20-10-0-result-table.csv}
\textsf{mk-20-10},500,250.84s,14,137,226.18s,8,146,205.73s,7,150,\textbf{181.77s},6,151
,1000,940.53s,18,267,897.17s,11,279,744.05s,11,283,\textbf{686.76s},8,291
,3000,0.54\%(1),24,750,\textbf{0.57\%(2)},12,765,0.67\%(1),11,774,0.67\%(0),8,809
,5000,\textbf{1.34\%(0)},19,1094,1.57\%(0),9,1111,1.47\%(0),7,1126,1.67\%(0),6,1140
\end{filecontents*}
\csvreader[column count=20, no head, late after line=\\, respect none]
        {tables/csv/split-ccmknap-20-10-0-result-table.csv}{1=\acol,2=\bcol,3=\ccol,4=\dcol,5=\ecol,6=\fcol,7=\gcol,8=\hcol,9=\icol,10=\jcol,11=\kcol,12=\lcol,13=\mcol,14=\ncol,15=\ocol,16=\pcol,17=\qcol,18=\rcol,19=\scol,20=\tcol}
        {\multirow{4}{*}{\acol} & \bcol & \ccol & \dcol & \ecol & \fcol & \gcol & \hcol & \icol & \jcol & \kcol & \lcol & \mcol & \ncol}
        \midrule
        \begin{filecontents*}{tables/csv/split-ccmknap-40-30-0-result-table.csv}
\textsf{mk-40-30},500,\textbf{1.60\%(0)},9,136,2.19\%(0),6,142,2.27\%(0),5,145,2.90\%(0),4,139
,1000,4.03\%(0),10,213,4.41\%(0),5,223,\textbf{4.01\%(0)},5,238,4.33\%(0),4,250
,3000,\textbf{6.14\%(0)},10,603,10.77\%(0),2,605,10.77\%(0),2,605,9.70\%(0),2,608
,5000,\textbf{5.72\%(0)},15,1008,11.47\%(0),2,1003,11.47\%(0),2,1003,11.47\%(0),2,1004
\end{filecontents*}
\csvreader[column count=20, no head, late after line=\\, respect none]
        {tables/csv/split-ccmknap-40-30-0-result-table.csv}{1=\acol,2=\bcol,3=\ccol,4=\dcol,5=\ecol,6=\fcol,7=\gcol,8=\hcol,9=\icol,10=\jcol,11=\kcol,12=\lcol,13=\mcol,14=\ncol,15=\ocol,16=\pcol,17=\qcol,18=\rcol,19=\scol,20=\tcol}
        {\multirow{4}{*}{\acol} & \bcol & \ccol & \dcol & \ecol & \fcol & \gcol & \hcol & \icol & \jcol & \kcol & \lcol & \mcol & \ncol}
        \bottomrule
      \end{tabular}}
    }
  \end{table}

%%% Local Variables:
%%% mode: latex
%%% TeX-master: "../main"
%%% End:

\revtwo{Tables~\ref{table:continuous_splits} and \ref{table:binary_splits} display the results for continuous and binary variables, respectively. In both cases, we fix~$\tau = 0.2$ since this setting is the most challenging according to our previous results. The results show that~$\Part_\beta$ performs better on instances that are solved by~$\Part_\final$ before the time limit is reached. These instances have few variables, constraints, and scenarios. In the continuous case, this occurs for larger values of~$\beta$. In the binary case, this occurs for smaller values of~$\beta$. On the other hand,~$\Part_\final$ outperforms~$\Part_\beta$ on large instances. This is because~$\Part_\beta$ performs fewer merge operations compared to~$\Part_\final$, regardless of the value of~$\beta$. When more merge operations are carried out the underlying CCSP stays tractable for a larger amount of iterations. In summary, we recommend using~$\beta = \SI{200}{\percent}$ on smaller instances and~$\beta = \SI{100}{\percent}$, \ie, applying $\Part_\final$, on larger instances.}

%%% Local Variables:
%%% mode: latex
%%% TeX-master: "../main"
%%% End:

\end{document}